\documentclass[bj,preprint]{imsart}

\usepackage{amssymb}
\usepackage{mathrsfs}		
\usepackage{amsthm}

\usepackage[norefs,nocites]{refcheck}

\RequirePackage[colorlinks=true,linkcolor=blue,citecolor=blue,pdfencoding=auto]{hyperref}
\usepackage[numbers]{natbib}
\usepackage{fontenc}
\usepackage{amsfonts}
\usepackage{dsfont}
\usepackage{color}
\RequirePackage[dvipsnames]{xcolor}
\usepackage{mathtools}
\usepackage[shortlabels]{enumitem}
\usepackage{tikz}
\usepackage{tcolorbox}
\usepackage[capitalize,nameinlink]{cleveref}
\usepackage[final,notref,notcite]{showkeys}
\usepackage{framed}
\usepackage{multirow}
\usepackage{tikz}
\usepackage{listings}
\tcbuselibrary{most}
\usepackage{bm}

\newlist{condition}{enumerate}{10}
\setlist[condition]{label*=({A}\arabic*)}
\crefname{conditioni}{condition}{conditions}
\Crefname{conditioni}{Condition}{Conditions}

\newlist{condition2}{enumerate}{10}
\setlist[condition2]{label*=(\arabic{section}.\arabic{subsection}.\roman*)}
\crefname{condition2i}{condition}{conditions}
\Crefname{condition2i}{Condition}{Conditions}

\newlist{condition3}{enumerate}{10}
\setlist[condition3]{label*=({B}\arabic*)}
\crefname{condition3i}{condition}{conditions}
\Crefname{condition3i}{Condition}{Conditions}

\colorlet{shadecolor}{orange!20}
\newtcolorbox{mybox}[1]{colback=red!5!white,
    colframe=red!75!black,fonttitle=\bfseries,
title=#1}

\crefname{equation}{}{}
\crefformat{equation}{(#2#1#3)}
\crefrangeformat{equation}{(#3#1#4)--(#5#2#6)}

\crefname{page}{p.}{pp.}
\Crefname{subsection}{Subsection}{Subsections}
\crefname{subsection}{subsection}{subsections}

\numberwithin{equation}{section}

\theoremstyle{plain}
\newtheorem{theorem}{Theorem}[section]

\newtheorem{lemma}[theorem]{Lemma}

\newtheorem{proposition}[theorem]{Proposition}

\theoremstyle{definition}
\newtheorem{remark}{Remark}[section]

\numberwithin{equation}{section}

\newcommand{\R}{\mathbb{R}}  
\newcommand{\ML}{\MoveEqLeft}         
\newcommand{\ignore}[1]{}{}

\newcommand{\tcref}[1]{\texorpdfstring{\cref{#1}}{#1}}

\DeclareMathOperator{\E}{ \mathbb{E} }
\DeclareMathOperator{\IProb}{ \mathbb{P} }
\newcommand{\IP}{\IProb}

\newcommand{\conep}[3][]{\E_{#1}\left\{ #2 \,\middle\vert\, #3  \right\}}
\newcommand{\conepb}[3][]{\E_{#1}\bigl\{ #2 \bigm\vert #3  \bigr\}}

\newcommand{\conepbc}[3][]{\E_{#1}\bigl\{ #2 \bigm\vert #3  \bigr\}}

\newcommand{\dd}{\mathop{}\!\mathrm{d}}

\renewcommand{\leq}{\leqslant}
\renewcommand{\geq}{\geqslant}

\newenvironment{equ}
{\begin{equation} \begin{aligned}}
{\end{aligned} \end{equation}}
\AfterEndEnvironment{equ}{\noindent\ignorespaces}
\newenvironment{equ*}
{\begin{equation*} \begin{aligned}}
{\end{aligned} \end{equation*}}
\AfterEndEnvironment{equ*}{\noindent\ignorespaces}
\def\cite{\citet*}
\def\ref{\cref}

\makeatletter
\newcommand{\refcheckize}[1]{%
  \expandafter\let\csname @@\string#1\endcsname#1%
  \expandafter\DeclareRobustCommand\csname relax\string#1\endcsname[1]{%
    \csname @@\string#1\endcsname{##1}\@for\@temp:=##1\do{\wrtusdrf{\@temp}\wrtusdrf{{\@temp}}}}%
  \expandafter\let\expandafter#1\csname relax\string#1\endcsname
}
\newcommand{\refcheckizetwo}[1]{%
  \expandafter\let\csname @@\string#1\endcsname#1%
  \expandafter\DeclareRobustCommand\csname relax\string#1\endcsname[2]{%
    \csname @@\string#1\endcsname{##1}{##2}\wrtusdrf{##1}\wrtusdrf{{##1}}\wrtusdrf{##2}\wrtusdrf{{##2}}}%
  \expandafter\let\expandafter#1\csname relax\string#1\endcsname
}
\makeatother

\newcommand{\vertiii}[1]{{\left\vert\kern-0.25ex\left\vert\kern-0.25ex\left\vert #1 \right\vert\kern-0.25ex\right\vert\kern-0.25ex\right\vert}}

\renewcommand{\tilde}{\widetilde}
\renewcommand{\ignore}[1]{}{}

\let\emptyset\varnothing    

\newcommand{\ex}[1]{e^{#1}}

\def\ML{\MoveEqLeft}

\newlist{con1}{enumerate}{10}
\setlist[con1]{label*=\normalfont{({LD}\arabic*)},ref=\normalfont{({LD}\arabic*)}}
\crefname{con1i}{condition}{conditions}
\Crefname{con1i}{Condition}{Conditions}

\newlist{con2}{enumerate}{10}
\setlist[con2]{label*=\normalfont{({A}\arabic*)},ref=\normalfont{({A}\arabic*)}}
\crefname{con2i}{condition}{conditions}
\Crefname{con2i}{Condition}{Conditions}

\newlist{con4}{enumerate}{10}
\setlist[con4]{label*=\normalfont{({B}\arabic*)},ref=\normalfont{({B}\arabic*)}}
\crefname{con4i}{condition}{conditions}
\Crefname{con4i}{Condition}{Conditions}

\newlist{con3}{enumerate}{10}
\setlist[con3]{label*=\normalfont{({LD}\arabic*)},ref=\normalfont{({LD}\arabic*)}}
\crefname{con3i}{condition}{conditions}
\Crefname{con3i}{Condition}{Conditions}

\renewcommand{\cite}{\citet*}

\refcheckize{\cref}
\refcheckize{\Cref}

\allowdisplaybreaks[4]
\begin{document}
\begin{frontmatter}
\title{Cram\'er-type moderate deviation of normal approximation for unbounded exchangeable pairs}
\runtitle{Cram\'er-type moderate deviation for exchangeable pairs}

\begin{aug}
\author[C]{\fnms{Zhuo-Song} \snm{Zhang}\ead[label=e3]{zszhang.stat@gmail.com}}
\runauthor{Z.-S.~Zhang}
\address[C]{
Department of Statistics and Applied Probability, 
		 National University of Singapore, 
		 Singapore 117546. \printead{e3}
}
\end{aug}



\begin{abstract}
    In Stein's method, the exchangeable pair approach is commonly used to estimate the approximation errors in normal approximation. 
    In this paper, we  establish a  Cram\'er-type moderate deviation theorem of normal approximation for unbounded exchangeable pairs. 
    As applications, Cram\'er-type moderate deviation theorems for the sums of local statistics and  general Curie--Weiss model are obtained. 
\end{abstract}

\end{frontmatter}
\maketitle

\section{Introduction}
\newcommand{\cJ}{\mathcal{J}}
\newcommand{\cA}{\mathcal{A}}
\newcommand{\cB}{\mathcal{B}}
\newcommand{\cC}{\mathcal{C}}
\newcommand{\bX}{{X}}
\newcommand{\card}[1]{\left\vert #1 \right\vert}

\newcommand{\red}{\color{red}} 

The {exchangeable pair approach} of Stein's method is commonly used to estimate the convergence rates for distributional approximation. Using exchangeable pair
approach, 
\cite{Ch11N} and \cite{Sh16I} provided a concrete tool to identify the limiting distribution
of the target random variable as well as the $L_{1}$ bound of the approximation. 
 Recently,  \cite{Sh19B} obtained a Berry--Esseen-type bound 
of normal and nonnormal approximation for unbounded exchangeable pairs. 
Specifically, let $W$ be the random variable of interest, and we say $(W, W')$ an  \emph{exchangeable pair} if $(W, W') \stackrel{d.}{=} (W', W)$.  
Let  $\Delta = W - W'$.  It is often to assume that (see, e.g., \cite{Ri97O}) there exists a constant $\lambda > 0$ and  a random variable $R$ such that 
\begin{equ}
    \label{eq:cex1} 
    \conep{\Delta}{W} = \lambda ( W + R). 
\end{equ}
Under condition \cref{eq:cex1}, \cite{Sh19B} proved the following Berry--Esseen-type bound
\begin{equ}
    \label{eq:cex2} 
    \sup_{z \in \mathbb{R}} \bigl\lvert \IP(W \leq z) - \Phi(z) \bigr\rvert 
    & \leq \E \biggl\lvert 1 - \frac{1}{2\lambda} \conepb{\Delta^2}{W} \biggr\rvert  + \frac{1}{\lambda} \E\bigl\lvert \conep{\Delta^{*}\Delta}{W} \bigr\rvert + \E |R|, 
\end{equ}
where $\Phi(z)$ is the standard normal distribution function and where $\Delta^{*}:= \Delta^{*}(W, W')$ is any random variable satisfying that $\Delta^{*} (W, W') = \Delta^{*}(W', W)$ and $\Delta^{*} \geq |\Delta|$. 
We refer to \cite{St86A,Ri97O,Ch05Ea,Cha08a} and \cite{Mec09} for other related results of $L_1$ bound and Berry--Esseen bound for the exchangeable pair approach.

While the $L_1$ bound and Berry--Esseen-type bound  
describe the absolute error for distributional approximations, 
the Cram\'er-type moderate deviation reflects the relative error in convergences in distribution. 
More precisely, let $\{Y_n, n \geq 1\}$ be a sequence of random variables that converge 
to $Y$ in distribution, the Cram\'er-type moderate deviation is  
\begin{align*}
    \frac{\IP(Y_n > x)}{ \IP(Y > x) } = 1 + \, \text{error term} \, \to 1 
\end{align*}
for $0 \leq x \leq a_n$, where $a_n \to \infty$ as $n \to \infty$. 
Specially, let $X_1, \dots, X_n$ be independent and identically distributed (i.i.d.) random variables  satisfying that $\E X_1 = 0$, $\E X_1^2 = 1$ and $\E \ex{ t_0 {|X_1|}} < \infty$ for some $t_0 > 0$, and put $W_n = n^{-1/2} (X_1 + \dots  + X_n)$. Then,  
\begin{equ}
    \label{eq:pe75} 
    \frac{\IP(W_n > x)}{ 1 - \Phi(x) } = 1 + O(1) n^{-1/2} (1 + x^3), 
\end{equ}
for $0 \leq x \leq n^{1/6}$. 
The range $0 \leq x \leq n^{1/6}$ and the order of the error term
$n^{-1/2}(1 + x^3)$ are optimal for i.i.d.\  random variables. 
We refer to Chapter 8 of \cite{Pe75S} for details. 

The proof of Cram\'er-type moderate deviation theorems for independent random variables is based on the conjugate method and the Fourier transform. However, when dealing with dependent random variables, it is more common to use Stein's method to estimate distributional
approximation errors.  
Since introduced by \cite{St72B} in 1972, Stein's method has been deeply developed in recent years, and shows its 
importance and power in estimating the approximation errors of normal and nonnormal approximation. 
We refer to \cite{Ch11N0} and \cite{Ch14S} for more details. 
Recent years have seen a rapid development of applying Stein's method to prove moderate deviation results.  
For instance, using Stein's method, \cite{Ra07C} proved  the moderate deviation under certain local dependence structures. 
In the context of Poisson approximation, \cite{Ba92P,Ch92S} and \cite{Ba95P} applied Stein's method to prove moderate deviation results for sums of independent indicators, whereas
\cite{Ch13M0} studied sums of dependent indicators. 
Moreover, \cite{Ch13S}, \cite{Sh18C} and \cite{Fa19S} obtained the general Cram\'er-type moderate 
deviation results of normal and nonnormal approximation for dependent 
random variables whose dependence structure is defined in terms of 
a Stein identity under a boundedness assumption on $|\Delta|$.  

However, in practice, it may not be easy to check the condition \cref{eq:cex1} in general, and the boundedness assumption on $\lvert \Delta \rvert$ is also too strict in applications.  
In this paper, our aim is to apply Stein's method and the exchangeable pair approach to prove a 
Cram\'er-type moderate deviation result without assuming that $|\Delta|$ is bounded.
The results are then applied to sums of local statistics and the general Curie--Weiss model to obtain the Cram\'er-type moderate deviation results with optimal ranges and convergence rates. 

The rest of this paper is organized as follows. We present our main results in Section 2. 
In Section 3, 
we give some applications of our main result.  
The proof of Theorem 2.1 are put in Section 4.  
The proofs of other results are  
postponed to Section 5. 

\section{Main results} 

Let $\bX$ be a random variable valued on a measurable space $\mathscr{X}$, and let $W = \varphi(\bX) \in \mathbb{R}$ be an $\R$-valued random variable of interest. 
We consider the following condition: 
\begin{enumerate}
    \item [(D1)] 
		Let $(\bX, \bX')$ be an exchangeable pair.  Assume that there exists $D \coloneqq \Psi(\bX ,\bX') $, where $\Psi : \mathscr{X} \times \mathscr{X} \to \R$ is an antisymmetric function, satisfying that $\conep{D}{\bX} = \lambda (W + R)$ for some constant $\lambda>0$ and
		some random variable $R \coloneqq
		r(X)$.
\end{enumerate}

\begin{remark}
	The operator of antisymmetric functions was introduced by \cite{Hol04S}, and the condition (D1) has been  considered by \cite{Ch07S}, who applied Stein's method to prove concentration inequalities. 
	The condition \textup{(D1)}  is a natural generalization of \cref{eq:cex1}. Specially, if \cref{eq:cex1} is satisfied, we can simply choose $D = \Delta$. 

	Under the condition \textup{(D1)}, by antysymmetry, it follows that $\E \bigl\{D \bigl(f (W) + f(W')\bigr)\bigr\} = 0$ for any absolutely continuous function $f : \R \to \R$ satisfying that $\E \lvert D f(W) \rvert < \infty$. 
    A direct rearranging yields  
    \begin{align*}
        0 & = \E \bigl\{D \bigl(f (W) + f(W')\bigr)\bigr\}\\
		  & = 2 \E \{D f(W)\} -  \E \bigl\{D \bigl( f(W) - f(W') \bigr)\bigr\} \\
          & = 2 \lambda \E \{(W + R) f(W)\} -  \E \biggl\{D \int_{-\Delta}^0 f'(W + u) \dd u\biggr\}.  
    \end{align*}
    Then, 
    \begin{equ}
        \E \bigl\{ W f(W)  \bigr\} = \frac{1}{2\lambda} \E \biggl\{D \int_{-\Delta}^0 f'(W + u) \dd u\biggr\} - \E \bigl\{ R f(W) \bigr\}.
        \label{eq:ewfw}
    \end{equ}
\end{remark}

Our main result \cref{thm:2.1} provides a Cram\'er-type moderate deviation theorem 
under the condition \textup{(D1)} without the assumption that $|\Delta|$ is bounded: 

\begin{theorem}
    \label{thm:2.1}    
	Let  $(X, X')$ be an exchangeable pair satisfying the condition \textup{(D1)}, let $W = \varphi(X)$, $W' = \varphi(X')$
    and $\Delta = W - W'$. 
	Let $D^{*} \coloneqq D^{*}({X}, {X}')$ be any random variable such that $D^{*} ({X}, {X}') = D^{*} ({X}', {X})$ and $D^{*} \geq |D|$.      	      
    Assume  that there exists a constant $\tau > 0 $ such that 
\begin{con2}
\item \label{c:a0} 
    \begin{math}
		\E \bigl\{(1 + |D|)\ex{t W}\bigr\} < \infty, 
    \end{math} 
\item \label{c:a1}
    \begin{math}
        \E\bigl\{\bigl\lvert 1 - \frac{1}{2\lambda} \conep{D\Delta}{\bX}\bigr\rvert \ex{tW}\bigr\} \leq \delta_1(t) \E \ex{tW}, 
    \end{math}
\item \label{c:a2}
    \begin{math}
        \E\bigl\{\bigl\lvert \frac{1}{2 \lambda} \conep{D^{*}\Delta}{\bX}\bigr\rvert \ex{tW}\bigr\} \leq \delta_2(t) \E \ex{tW}, 
    \end{math}
    and 
\item \label{c:a3}
    \begin{math}
        \E\bigl\{|R| \ex{tW}\bigr\} \leq \delta_3(t) \E \ex{tW}.  
    \end{math}
\end{con2}
where for each $j = 1, 2, 3$, the function $\delta_j(\cdot)$ is increasing and satisfies that $\delta_j(\tau) < \infty$.  
    For $\theta > 0$, let 
    \begin{math}
        \tau_0(\theta)  := \max \{ 0 \leq t \leq \tau : t^2 \bigl(\delta_1(t) + \delta_2(t)\bigr)/2 + t \delta_3(t) \leq \theta  \}. 
    \end{math}
    Then, for any $\theta > 0$,  
    \begin{equ}
        \label{eq:lemexch}
        \biggl\lvert \frac{\IP (W > z)}{ 1 - \Phi(z)} - 1 \biggr\rvert  & \leq 
        20 \ex{\theta}   \bigl( (1 + z^2) \bigl( \delta_1 (z) + \delta_2(z)\bigr) + (1 + z) \delta_3(z) \bigr), 
    \end{equ}
    provided that $0 \leq z \leq \tau_0(\theta) $. 
\end{theorem}

\begin{remark}
	Recently, \cite{Zh21B} proved the following Berry--Esseen bound under the condition (D1): 
	\begin{equ}
    \sup_{z \in \mathbb{R}} \bigl\lvert \IP(W \leq z) - \Phi(z) \bigr\rvert 
    & \leq \E \biggl\lvert 1 - \frac{1}{2\lambda} \conepb{ D \Delta}{X} \biggr\rvert  + \frac{1}{\lambda} \E\bigl\lvert \conep{D^{*}\Delta}{X} \bigr\rvert + \E |R|. 	
        \label{eq-r2.2}
	\end{equ}
	The expectation terms in \cref{c:a1,c:a2,c:a3} can be understood as a change of measure of those on the right hand side of \cref{eq-r2.2}. To see this, for $t \geq 0$, let $Y_t$ be a random variable having the distribution 
	\begin{align*}
		\IP ( Y_t \in A ) = \frac{ \E \bigl\{e^{ t \varphi(X) } \mathbf{1}_{ \{X \in A \} } \bigr\}
		}{\E e^{t \varphi(X)}}.
	\end{align*}
	Let $g_1$ and $g_2$ be functions defined as 
	\begin{align*}
		g_1(x) = 1 - \frac{1}{2\lambda} \E \{ D \Delta \vert X = x \}, \quad 
		g_2(x) = \frac{1}{2\lambda} \E \{ D^* \Delta \vert X = x \},
	\end{align*}
	and recall that $R = r(X)$ as given in (D1). 
	Then, \cref{c:a1,c:a2,c:a3} can be replaced by 
	\begin{align*}
		\E \lvert g_1(Y_t) \rvert \leq \delta_1(t), \quad 
		\E \lvert g_2(Y_t) \rvert \leq \delta_2(t),	\quad  
		\E \lvert r(Y_t) \rvert \leq \delta_3(t).
	\end{align*}

\end{remark}

\section{Applications}
\subsection{Sums of local statistics}
Let $\mathcal{J}$ be an index set and let $X = \{ X_{\alpha} : \alpha \in \mathcal{J} \}$ be a field of independent random variables where $X_{\alpha}$ is valued on a measurable space $\mathcal{X}$. For any subset $J \subset \mathcal{J}$, denote $X_J = \{ X_{\alpha} : {\alpha} \in J \}$.
Let $n \geq 1$ and let $[n] = \{ 1, \dots, n \}$. For each $i \in [n]$, let 
\begin{math}
	\xi_i = f_i (X_{J_i}),
\end{math}
where $J_i \subset \mathcal{J}$ and $f_i : \mathcal{X}^{ \lvert J_i \rvert } \to \R$,  satisfying that $\E \xi_i = 0$ for each $i \in [n]$ and $\sum_{i = 1}^n \E \xi_i^2 = 1$. 
Let $N_{\alpha} = \{ i \in [n]: {\alpha} \in J_i \}$ for ${\alpha} \in \mathcal{J}$. For any set $A$, let $\lvert A \rvert$ be its cardinality. 
We consider the sum of local statistic $W = \sum_{i = 1}^n \xi_i$.

\begin{theorem}	
	\label{thm:local}
Assume that 
for each $i \in [n]$, 
\begin{align}
	\lvert \xi_i \rvert \leq \sum_{{\alpha} \in J_i} g_{i{\alpha}} (X_{\alpha}), \label{eq:c31}
\end{align}
where $\{ g_{i{\alpha}} : i \in [n], {\alpha} \in J_i \}$ is an array of nonnegative functions, 
and there exist $a > 0$ and $b \geq 1$ such that 
\begin{align}
	\max_{i \in [n]} \max_{{\alpha} \in J_i} \E e^{ a g_{i{\alpha}}(X_{\alpha}) } \leq b. 
	\label{eq:c32}
\end{align}
Moreover, we further assume that there exist $s \geq 1$ and $d \geq 1$ such that 
\begin{equ}
	 \max_{i \in [n]} \lvert J_i \rvert \leq s, \quad \max_{{\alpha} \in \mathcal{J}} \lvert N_{\alpha} \rvert \leq d. 
    \label{eq:c33}
\end{equ}
Let $\delta = n^{1/2} a^{-2} d^{3/2} s^{7/2} b^{1/2} (1 + n^{1/2} a^{-1} d^{1/2} s^{3/2} b^{1/2}).$
Then, we have 
\begin{equ}
	\biggl\lvert \frac{ \IP(W \geq z) }{1 - \Phi(z)} - 1 \biggr\rvert 
	& \leq C_0 \delta (1 + z^3) 
    \label{eq:t3.1-a}
\end{equ}
for $0 \leq z \leq \min\{ a/(8dsb^2), \delta^{-1/3}\}$, 
where $C_0$ is an absolute positive constant. 
\end{theorem}
\begin{remark}
	We make some remarks on the conditions in \cref{thm:local}.
	The condition \cref{eq:c31} can be satisfied for a wide class of functions. For example, if $f_i(x_{J_i}) = \prod_{j\in J_i} x_{j}$, then \cref{eq:c31} is satisfied with $g_{ij} = \lvert x_j \rvert^{ \lvert J_i \rvert }/ \lvert J_i \rvert$ by Young's
	inequality.  Condition \cref{eq:c32} is also known as the \emph{Cram\'er's condition}, which is commonly taken when proving Cram\'er-type moderate deviation theorems. 
	The condition \cref{eq:c33} is also assumed by \cite{Fa19S}. We remark that both $d$ and $s$ in \cref{eq:c33} can also depend on $n$. 
\end{remark}
\begin{remark}
	In order to illustrate the range and convergence rate \cref{eq:t3.1-a} are correct, consider sums of i.i.d.\  random variables. 
	Let $X_1, \dots, X_n$ be i.i.d.\  random variables satisfying that $\E X_i = 0, \E X_i^2 = 1$ and $\E e^{t_0 \lvert X_i \rvert} \leq b$
	for some $t_0 > 0$ and $b \geq 1$. Let $\xi_i = X_i / \sqrt{n}$ for $i \in [n]$ and $W = \sum_{i = 1}^n \xi_i$. Then, we have \cref{eq:c31,eq:c32,eq:c33} are satisfied with
	\begin{align*}
		a = t_0 \sqrt{n}, \quad s = d = 1.  
	\end{align*}
	Then $\delta = n^{-1/2} t_0^{-2} b^{1/2}(1 + t_0^{-1} b^{1/2}) \leq n^{-1/2}(1 + t_0^{-3} b)$.
	Therefore, \cref{thm:local} reduces to 
	\begin{equation*}
		\biggl\lvert \frac{ \IP(W \geq z) }{ 1 - \Phi(z) } - 1 \biggr\rvert
		\leq C_0 n^{-1/2} (1 + t_0^{-3}b) (1 + z^3).	
	\end{equation*}
	$\text{for $0 \leq z \leq \min \{ t_0 b^{-2} n^{1/2}, t_0^{2/3} b^{-1/6} n^{1/6} \} $}$, which is as same as the optimal result \cref{eq:pe75}.
\end{remark}	
\begin{remark}
	Recently, \cite{Fa19S} proved a higher-order approximation relative error for the cases where $\xi_i$'s are bounded. Although we only consider low-order approximation in this subsection, the boundedness assumption is relaxed in \cref{thm:local}. It would be
	interesting if one could prove a higher-order approximation error bound for unbounded cases. 
\end{remark}

\subsection{The general Curie--Weiss model}
The Curie--Weiss model of ferromagnetic interaction has been extensively 
studied in the past decades. Let $\rho$ be a probability measure on $\mathbb{R}$ 
satisfying that 
\begin{equ}
    \label{eq:cw-c1} 
    \int_{-\infty}^{\infty} x \dd \rho(x) = 0 \quad \text{and}\quad \int_{-\infty}^{\infty} x^2 \dd \rho(x) = 1.
\end{equ}
The {\it general Curie--Weiss model} ${\rm CW}(\rho)$ at inverse temperature $\beta$ is defined as the array of spin random variables 
${X} = (X_1, \dots, X_n)$ with joint distribution 
\begin{equ}
    \label{eq:cw-c2} 
    \dd P_{n, \beta} ({x}) = Z_n^{-1} \exp \biggl( \frac{\beta}{2n} (x_1 + \dots + x_n)^2 \biggr) \prod_{i = 1}^n \dd \rho(x_i)
\end{equ}
for ${x} = (x_1, \dots, x_n) \in \mathbb{R}^n$ where $Z_n$ is the normalizing constant
\begin{equ*}
    Z_n = \int \exp \biggl( \frac{\beta}{2n} (x_1 + \dots + x_n)^2 \biggr) \prod_{i = 1}^n \dd \rho(x_i) . 
\end{equ*}
The Curie--Weiss model ${\rm CW}(\rho)$ is called ``at the critical temperature'' if $\beta = 1$. 
The total magnetization is defined by 
\begin{math}
    S = \sum_{i = 1}^{n} X_i.
\end{math}
The asymptotic behavior of the $S$ is well studied by \cite{El78L,El78T}. 
Stein's method can be applied to estimate the convergence rate, for example, using the exchangeable pair approach, \cite{Ei10S,Ch11N} obtained Berry--Esseen bounds 
for Curie--Weiss model with boundedly supported $\rho$, and \cite{Sh19B} proved the Berry--Esseen bound for the general Curie--Weiss model with unboundedly supported $\rho$. We refer to \cite{Ki13A,Ki16A,Ei15O,Ei16R} for other normal approximation results of
mean field models by Stein's method. 
Moreover, \cite{Ch13S} and \cite{Sh18C} obtained the Cram\'er-type moderate deviation results for the cases where $\rho$ has a finite support. 

In this subsection, we establish the Cram\'er-type moderate deviation result for the general Curie--Weiss model at noncritical temperature with infinitely supported probability measure $\rho$. 
Let $(X_1, \dots, X_n)$ follow the joint distribution \cref{eq:cw-c2} with $0 < \beta < 1$ and $\rho$ satisfying \cref{eq:cw-c1} and 
\begin{equ}
    \label{eq:cw-c4} 
    \int_{-\infty}^{\infty} \ex{t x} \dd \rho(x) \leq \ex{t^2/2} \quad \text{for all } t \in \mathbb{R}.
\end{equ}
Let $W = n^{-1/2}(1 - \beta)^{1/2} \sum_{i = 1}^n X_i$ be the standardized version of the total magnetization. We have the following theorem. 

\begin{theorem}
        \label{thm:cw}
		Under \cref{eq:cw-c1,eq:cw-c4},  
        we have 
        \begin{equ}
            \label{eq:tcw-a} 
            \biggl\lvert \frac{\IP(W> z)}{1 - \Phi(z)} - 1 \biggr\rvert \leq C n^{-1/2} (1 + z^3)
        \quad \text{ for $0 \leq z \leq \sqrt{n}$}.
        \end{equ}
\end{theorem}

The Berry--Esseen bound was obtained by \cite{Sh19B} with the convergence rate $O(n^{-1/2})$. 
For the simplest Curie--Weiss model, where the magnetization is valued on $\left\{ -1,1 \right\}$ with equal probability, 
\cite{Ch13S} proved the same convergence rate as \cref{eq:tcw-a} with convergence range $[0, n^{1/6}]$. However, \cref{thm:cw} provides a wider 
convergence range.

\section{Proof of main result}
In this section, we give the proof of \cref{thm:2.1}. In \cref{sub:4.1}, we prove a more general moderate deviation result, which might be of independent interest. In \cref{sub:4.2}, we prove a bound for $\E e^{tW}$ using Stein's method. Our main result \cref{thm:2.1} follows from
a combination of \cref{thm:exch0,prop:5.3}, and the details of the proof are put in \cref{sub:4.3}.

\subsection{A general moderate deviation result}%
\label{sub:4.1}

\def\tdelta{\kappa}
\begin{proposition}
    \label{thm:exch0}
	Let  $(X, X')$, $W$, $W'$, $D, \Delta$ and $D^*$ be defined as in \cref{thm:2.1}.      
    Assume  that there exists a constant $\tau_0 > 0 $  such that for all $0 \leq t \leq \tau_0$, 
\begin{con4}
\item \label{c:c1}
    \begin{math}
        \E\bigl\{\bigl\lvert 1 - \frac{1}{2\lambda} \conep{D\Delta}{\bX}\bigr\rvert \ex{tW}\bigr\} \leq \tdelta_1(t)\ex{t^2/2} , 
    \end{math}
\item \label{c:c2}
    \begin{math}
        \E\bigl\{\bigl\lvert \frac{1}{2 \lambda} \conep{D^{*}\Delta}{\bX}\bigr\rvert \ex{tW}\bigr\} \leq \tdelta_2(t)\ex{t^2/2}, 
    \end{math}
    and 
\item \label{c:c3}
    \begin{math}
        \E\bigl\{|R| \ex{tW}\bigr\} \leq \tdelta_3(t) \ex{t^2/2}.  
    \end{math}
\end{con4}
where $\tdelta_1(t), \tdelta_2(t)$ and $\tdelta_3(t)$ are nondecreasing functions satisfying that $\tdelta_j(\tau_0) < \infty$ for $j = 1, 2, 3$. 
Then,   
    \begin{equ}
        \label{eq:lemexch0}
        \biggl\lvert \frac{\IP (W > z)}{ 1 - \Phi(z)} - 1 \biggr\rvert  & \leq 
        20   \bigl( (1 + z^2) \bigl( \tdelta_1 (z) + \tdelta_2(z)\bigr) + (1 + z) \tdelta_3(z) \bigr), 
    \end{equ}
    provided that $0 \leq z \leq \tau_0 $. 
\end{proposition} 
We now make some remarks on \cref{thm:2.1} and \cref{thm:exch0}. In practice, if we know little about the explicit distributional information of the random variable $X$, it might not be easy to obtain the $e^{t^2/2}$ term directly when calculation the expectation terms on the left
hand of \cref{c:c1,c:c2,c:c3}. However, on the other hand, it is usually easier to obtain the self-bounding inequalities \cref{c:a1,c:a2,c:a3} only based on the structure of $W$. After that, one can apply Stein's method to prove the moment generating inequality $\E e^{tW} \leq C e^{t^2/2}$, which
further
implies \cref{c:c1,c:c2,c:c3}. We refer to
\cref{prop:5.3} (see below) for more details for the moment generating inequality. 

Before proving \cref{thm:exch0}, we first present some preliminary lemmas. 
In the proofs, we use the techniques in \cite[Lemmas 5.1--5.2]{Ch13S} and \cite[pp. 71--73]{Sh19B}. 

\begin{lemma}
    \label{lem:psi}
    Let $f$ be a nondecreasing function. Then, 
    \begin{equ}
        \biggl\lvert \E \biggl\{ D \int_{-\Delta}^0 \bigl(  f(W + u) - f(W) \bigr) \dd u  \biggr\}\biggr\rvert \leq \E \bigl\{ D^{*} \Delta f(W) \bigr\}, 
        \label{l4.1-aa}
    \end{equ}
    where $D^*$ is as defined in \cref{thm:exch0}. 
\end{lemma}

\begin{proof}
    [Proof of \cref{lem:psi}] 
\def\func{f}
\newcommand\IN[1]{\mathbf{1}_{ \{#1\} }}
	In this proof, we use the technique as in \cite{Sh19B}. 
    Since $\func(\cdot)$ is nondecreasing, it follows that
	$\Delta \bigl( \func(W) - \func(W') \bigr) \geq 0$
	and that 
    \begin{equ}
        0 & \geq \int_{-\Delta}^0 \bigl( \func(W + u) - \func(W) \bigr) \dd u \geq - \Delta \bigl( \func(W) - \func(W') \bigr). 
        \label{l4.1-01}
    \end{equ}
	Let $D^+ = D \IN{ D > 0 }$ and $D^- = -D\IN{D < 0}$. Then, it follows that $D = D^+ - D^-$ and $|D| = D^+ + D^-$. By \cref{l4.1-01}, we have 
	\begin{align*}
		D^+\int_{-\Delta}^0 \bigl( \func(W + u) - \func(W) \bigr) \dd u \leq 0, \quad 
		D^-\int_{-\Delta}^0 \bigl( \func(W + u) - \func(W) \bigr) \dd u \leq 0.
	\end{align*}
	Therefore, 
    \begin{align*}
		\E \biggl\{ D \int_{-\Delta}^0 \bigl( \func(W + u) - \func(W) \bigr) \dd u \biggr\}
		& = \E \biggl\{ (D^+ - D^-) \int_{-\Delta}^0 \bigl( \func(W + u) - \func(W) \bigr) \dd u \biggr\}\\
		& \geq \E \biggl\{ D^+ \int_{-\Delta}^0 \bigl( \func(W + u) - \func(W) \bigr) \dd u \biggr\}\\
		& \geq - \E \bigl\{ D^+ \Delta \bigl( f(W) - f(W') \bigr) \bigr\}.
    \end{align*}
	and similarly 
	\begin{align*}
		\E \biggl\{ D \int_{-\Delta}^0 \bigl( \func(W + u) - \func(W) \bigr) \dd u \biggr\}
		& \leq \E \bigl\{ D^- \Delta \bigl( f(W) - f(W') \bigr) \bigr\}.
	\end{align*}
    Recalling that $ W= \phi({X})$, $D = F({X}, {X}')$ is antisymmetric and $D^* = F^*({X}, {X}')$ is symmetric, as $({X}, {X}')$ is exchangeable,  we have 
    \begin{gather*}
        \E \bigl\{ D^+ \Delta \bigl\{ \func(W) - \func(W') \bigr\} \bigr\} = \E \bigl\{ D^- \Delta \bigl( \func(W) - \func(W') \bigr) \bigr\} \geq 0. 
    \end{gather*}
	Therefore, the LHS of \cref{l4.1-aa} is bounded by
	\begin{equ}
		\MoveEqLeft \frac{1}{2\lambda} \biggl\lvert \E \biggl\{ D \int_{-\Delta}^0 \bigl\{ \func(W + u)  - \func(W) \bigr\} \dd u\biggr\}\biggr\rvert \leq \frac{1}{2\lambda} \E \bigl\{ D^- \Delta \bigl( \func(W) - \func(W') \bigr)\bigr\}.
        \label{l4.1-02}
	\end{equ}
    In order to bound the RHS of \cref{l4.1-02}, note that $\E \bigl\{D^* \Delta \IN{ D = 0} \bigl( \func(W) - \func(W') \bigr)\bigr\} \geq 0$ and that $\E \bigl\{ D^* \IN{ D = 0 } \Delta \func(W) \bigr\} = - \E \bigl\{ D^* \IN{ D = 0 } \Delta \varphi(W') \bigr\}$,
	and then we have  
    \begin{equ}
        \E \{D^* \Delta \IN{ D = 0} \func(W)\} \geq 0. 
        \label{eq-l4.2-11}
    \end{equ}
	Moreover, note that $D^*$ is symmetric with respect to $X$ and $X'$, and by by exchangeability, 
	\begin{equ}
		\E \bigl\{D^*  \Delta \IN{ D < 0 } \func(W') \bigr\}
		= \E \bigl\{D^{*} \Delta \IN{ D > 0 }  \func(W)\bigr\}.
        \label{eq-l4.2-12}
	\end{equ}
    Therefore, by \cref{eq-l4.2-11,eq-l4.2-12}, 
    \begin{align*}
		\text{RHS of \cref{l4.1-02}}
        & \leq \frac{1}{2\lambda} \E \bigl\{D^* \IN{ D < 0 } \Delta \bigl( \func(W) - \func(W') \bigr) \bigr\}\\
        & = \frac{1}{2\lambda} \E \bigl\{D^{*} \Delta \bigl( \IN{ D > 0 } + \IN{ D <0 } \bigr) \func(W)\bigr\} \\
		& \leq \frac{1}{2\lambda} \E \{D^{*} \Delta \func(W)\}. 
	\end{align*}
	This completes the proof. 
\end{proof}

\begin{lemma}
    \label{lem:j2}
    Under the conditions of \cref{thm:exch0}, 
    we have for $0 \leq z \leq \tau_0 $, 
    \begin{align}
        \label{eq:lj1} 
		\E \left\{ \biggl\lvert 1 - \frac{1}{2\lambda} \conep{D \Delta}{W}  \biggr\rvert W \ex{W^2/2} \mathbf{1}_{ \left\{ 0 \leq W \leq z \right\} } \right\} & \leq 4   (1 + z^2) \tdelta_1(z) , 
        \\ 
        \label{eq:lj2} 
		\frac{1}{2\lambda} \E \left\{ \bigl\lvert \conep{D^{*} \Delta}{W}  \bigr\rvert W \ex{W^2/2} \mathbf{1}_{ \left\{ 0 \leq W \leq z \right\} } \right\} & \leq 4   (1 + z^2) \tdelta_2(z), 
        \\
		\E \left\{ |R| \ex{ W^2/2 } \mathbf{1}_{ \left\{ 0 \leq W \leq z \right\} } \right\} & \leq  2  (1 + z) \tdelta_3(z). 
        \label{eq:lj3} 
    \end{align}
\end{lemma}
\begin{proof}
    [Proof of \cref{lem:j2}]
    We apply the idea of \cite[Lemma 5.2]{Ch13S} in this proof. 
    For $a \in \mathbb{R}_+$, denote $\lfloor a \rfloor = \max\{ n \in \mathbb{N} : n \leq a \}$. 
    By \cref{c:c1}, and recalling that the function $\tdelta_1(\cdot)$ is increasing, for any $0 \leq x \leq z \leq \tau_0 $, 
    \begin{equ}
        \ex{-x^2/2}\E \left\{ \biggl\lvert 1 - \frac{1}{2\lambda} \conep{D \Delta}{W}  \biggr\rvert  \ex{ x W } \right\} & \leq \tdelta_1(x) \leq  \tdelta_1 (z). 
        \label{eq-lj201}
    \end{equ}
    We have  
        \begin{align*}
        \text{RHS of \cref{eq:lj1}}
        & = \sum_{j = 1}^{ \lfloor {z} \rfloor } \E \left\{ \biggl\lvert 1 - \frac{1}{2\lambda} \conep{D \Delta}{W}  \biggr\rvert W \ex{W^2/2} \mathbf{1}_{ \left\{ j-1 \leq W < j \right\} } \right\} \\
        & \quad +  \E \left\{ \biggl\lvert 1 - \frac{1}{2\lambda} \conep{D \Delta}{W}  \biggr\rvert W \ex{W^2/2} \mathbf{1}_{ \left\{ \lfloor {z} \rfloor \leq W \leq z \right\} } \right\} \\
        & \leq \sum_{j = 1}^{\lfloor {z} \rfloor} j \ex{ (j - 1)^2 / 2 - j ( j - 1 ) }   \E \left\{ \biggl\lvert 1 - \frac{1}{2\lambda} \conep{D \Delta}{W}  \biggr\rvert  \ex{ j W } \mathbf{1}_{ \left\{ j-1 \leq W < j \right\} } \right\} \\
        & \quad +  z \ex{ \lfloor {z} \rfloor^2/2 - \lfloor {z} \rfloor z } \E \left\{ \biggl\lvert 1 - \frac{1}{2\lambda} \conep{D \Delta}{W}  \biggr\rvert \ex{ z W } \mathbf{1}_{ \left\{ \lfloor {z} \rfloor \leq W \leq z \right\} } \right\} \\
        & \leq 2 \sum_{j = 1}^{\lfloor {z} \rfloor} j  \ex{-j^2/2} \E \left\{ \biggl\lvert 1 - \frac{1}{2\lambda} \conep{D \Delta}{W}  \biggr\rvert  \ex{ j W } \mathbf{1}_{ \left\{ j-1 \leq W < j \right\} } \right\} \\
        & \quad +  2 z \ex{-z^2/2} \E \left\{ \biggl\lvert 1 - \frac{1}{2\lambda} \conep{D \Delta}{W}  \biggr\rvert \ex{ z W } \mathbf{1}_{ \left\{ \lfloor {z} \rfloor \leq W \leq z \right\} } \right\}\\
		& \leq 2 \tdelta_1 (z)\biggl( \sum_{j = 1}^{\lfloor {z} \rfloor} j + z \biggr) \leq 4   (1 + z^2) \tdelta_1(z) ,
    \end{align*}
	where we used \cref{eq-lj201} in the last line. 
    This proves \cref{eq:lj1}. The inequalities \cref{eq:lj2,eq:lj3} can be shown similarly. 
\end{proof}
Now we are ready to give the proof of \cref{thm:exch0}. In the proof, we combine the ideas in \cite[Section 4]{Sh19B} and \cite[Section 6]{Ch13S}.
\begin{proof}
    [Proof of \cref{thm:exch0}] 
    Let $z \geq 0$ be a fixed real number,  
    and let $f_z$ be the solution to the Stein equation 
    \begin{align}
        f'(w) - w f(w) = \mathbf{1}_{ \left\{ w \leq z \right\} } - \Phi(z), 
        \label{eq:stein}
    \end{align}
    where $\Phi(\cdot)$ is the standard normal distribution function. 
    It is well known that (see, e.g., \cite{Ch11N0}) $f_z$ is given by 
    \begin{equ}
        \label{eq:solu}
        f_z (w) = 
        \begin{dcases}
            \frac{\Phi(w) \bigl\{ 1 - \Phi(z) \bigr\}}{ p(w) }, & w \leq z, \\
            \frac{ \Phi(z) \bigl\{ 1 - \Phi(w) \bigr\} }{ p(w) }, & w > z, 
        \end{dcases}
    \end{equ}
    where $p(w) = (2\pi)^{-1/2} \ex{ - w^2 / 2 }$ is the standard normal density function.  
    
	By \cref{eq:stein} and applying \cref{eq:ewfw} by taking $f(w) = f_z(w)$ yields 
    \begin{equ}
        \label{eq:j123}
        \IP (W > z) - \bigl\{ 1 - \Phi(z) \bigr\} = \E \bigl\{ f_z'(W) - W f_z(W) \bigr\} 
         =  J_1 - J_2 + J_3, 
    \end{equ}
    where 
    \begin{align*}
        J_1 & = \E \biggl\{ f_z'(W) \left( 1 - \frac{1}{2\lambda} \conep{ D \Delta}{W} \right) \biggr\}, \\
        J_2 & = \frac{1}{2\lambda} \E \biggl\{ D \int_{-\Delta}^0 \left( f_z'(W + u) - f_z'(W)  \right) \dd u  \biggr\} , \\
        J_3 & = \E \bigl\{ R f_z(W) \bigr\} .
    \end{align*}

    Without loss of generality, we only consider $J_2$, because $J_1$ and 
    $J_3$ can be bounded similarly. 

    For $J_2,$ observe that $f_z'(w) = w f(w) - \mathbf{1}_{ \left\{ w > z \right\} } + \bigl\{ 1 - \Phi(z) \bigr\}$, and both $wf_z(w)$ and $\mathbf{1}_{ \left\{ w > z \right\} }$ are increasing functions (see, e.g. \cite[Lemma 2.3]{Ch11N0}), by \cref{lem:psi}, 
    \begin{equ}
        \label{eq:j2122} 
        |J_2| &\leq \frac{1}{2\lambda} \biggl\lvert \E \biggl[ D \int_{-\Delta}^0 \left\{ (W + u)f_z(W + u) - W f_z'(W)  \right\} \dd u  \biggr]\biggr\rvert \\
              & \quad +  \frac{1}{2\lambda} \biggl\lvert \E \biggl[ D \int_{-\Delta}^0 \left\{ \mathbf{1}_{ \left\{ W + u > z \right\}  } - \mathbf{1}_{ \left\{ W > z \right\} } \right\} \dd u  \biggr]\biggr\rvert \\
              & \leq \frac{1}{2\lambda} \E \bigl\lvert \conep{D^{*} \Delta}{W} \bigr\rvert \bigl( |W f_z(W)| + \mathbf{1}_{ \left\{ W > z \right\} } \bigr) = J_{21} + J_{22}, 
    \end{equ}
    where 
    \begin{align*}
		J_{21} & =  \frac{1}{2\lambda} \E \Bigl\{\bigl\lvert \conep{D^{*} \Delta}{W} \bigr\rvert \cdot \bigl\lvert W f_z(W) \bigr\rvert\Bigr\}, & 
        J_{22} & = \frac{1}{2\lambda} \E \Bigl\{\bigl\lvert \conep{D^{*} \Delta}{W} \bigr\rvert \mathbf{1}_{ \left\{ W > z \right\} }\Bigr\} . 
    \end{align*}   
    For any $w > 0$, it is well known that
    \begin{math}
        ( 1 - \Phi(w) )/ p(w)  \leq \min \left\{ 1/w, \sqrt{2\pi}/2 \right\}. 
    \end{math}
    Then,  for $w > z$,  
    \begin{align}
        \label{eq:bound1}
        \bigl\lvert f_z(w) \bigr\rvert \leq \frac{\sqrt{2\pi}}{2} \Phi(z) , \quad \bigl\lvert w f_z(w) \bigr\rvert \leq \Phi(z), 
    \end{align}
    and by symmetry, for $w < 0$, 
    \begin{align}
        \label{eq:bound2}
        \bigl\lvert f_z(w) \bigr\rvert \leq \frac{\sqrt{2 \pi}}{2}  \bigl\{1 - \Phi(z)\bigr\}  , \quad \bigl\lvert w f_z(w) \bigr\rvert \leq 1 - \Phi(z).
    \end{align}

    For $J_{21},$ by \cref{eq:solu,eq:bound1,eq:bound2}, we have 
    \begin{equ}
        \label{eq:j2100}
        J_{21} & \leq  \frac{1}{2\lambda} \bigl\{1 - \Phi(z)\bigr\} \E \Bigl\{\bigl\lvert \conep{D^{*} \Delta}{W} \bigr\rvert \mathbf{1}_{ \left\{ W < 0 \right\} }\Bigr\} \\
               & \quad +  \frac{\sqrt{2\pi}}{2\lambda} \bigl\{1 - \Phi(z)\bigr\} \E \Bigl\{\bigl\lvert \conep{D^{*} \Delta}{W} \bigr\rvert W \ex{ W^2/2 } \mathbf{1}_{ \left\{ 0 \leq W \leq z \right\} }\Bigr\} \\
               & \quad +  \frac{1}{2\lambda} \E \Bigl\{\bigl\lvert \conep{D^{*} \Delta}{W} \bigr\rvert \mathbf{1}_{ \left\{ W > z \right\} }\Bigr\}. 
    \end{equ}
    Thus, by \cref{eq:j2122,eq:j2100}, 
    \begin{equ}
        \label{eq:pj20} 
        |J_2| & \leq \frac{1}{2\lambda} \bigl\{1 - \Phi(z)\bigr\} \E \Bigl\{\bigl\lvert \conep{D^{*} \Delta}{W} \bigr\rvert \mathbf{1}_{ \left\{ W < 0 \right\} }\Bigr\} \\
               & \quad +  \frac{\sqrt{2\pi}}{2\lambda} \bigl\{1 - \Phi(z)\bigr\} \E \Bigl\{\bigl\lvert \conep{D^{*} \Delta}{W} \bigr\rvert W \ex{ W^2/2 } \mathbf{1}_{ \left\{ 0 \leq W \leq z \right\} }\Bigr\} \\
               & \quad +  \frac{1}{\lambda} \E \Bigl\{\bigl\lvert \conep{D^{*} \Delta}{W} \bigr\rvert \mathbf{1}_{ \left\{ W > z \right\} }\Bigr\}. 
    \end{equ}
    For the second term of \cref{eq:pj20}, by \cref{lem:j2}, we have 
    \begin{equ}
        \label{eq:pj22} 
        \ML \frac{\sqrt{2\pi} }{2\lambda} \E \Bigl\{\bigl\lvert \conep{D^{*} \Delta}{W} \bigr\rvert W \ex{ W^2/2 } \mathbf{1}_{ \left\{ 0 \leq W \leq z \right\} }\Bigr\} \leq 4 \sqrt{2\pi}    (1 + z^2) \tdelta_2(z).
    \end{equ}
    It is well known that for $z > 0$,  
    \begin{align*}
        \ex{-z^2/2} \leq \sqrt{2\pi} ( 1 + z ) \bigl\{ 1 - \Phi(z) \bigr\} \leq \frac{3 \sqrt{2 \pi}}{2}( 1 + z^2 ) \bigl\{ 1 - \Phi(z) \bigr\} . 
    \end{align*}
    For the third term of \cref{eq:pj20},  
    by \cref{c:c2}, 
    for $0 \leq z \leq \tau_0 $, 
    \begin{equ}
        \label{eq:pj23} 
        \frac{1}{\lambda} \E \Bigl\{\bigl\lvert \conep{D^{*} \Delta}{W} \bigr\rvert \mathbf{1}_{ \left\{ W > z \right\} }\Bigr\} &  \leq 2\tdelta_2(z) \ex{-z^2/2} \leq  3  \sqrt{2 \pi} (1 + z^2) \tdelta_2(z) \bigl\{ 1 - \Phi(z) \bigr\}.
    \end{equ}
    Therefore, combining \cref{eq:pj20,eq:pj22,eq:pj23}, for $0 \leq z \leq \tau_0$, we have 
    \begin{align*}
		|J_2| & \leq (7 \sqrt{2\pi} + 1 )   (1 + z^2) \tdelta_2(z) \bigl( 1 - \Phi(z) \bigr) \leq 20 (1 + z^2) \tdelta_2(z) \bigl( 1 - \Phi(z) \bigr). 
    \end{align*}
    Similarly, 
    \begin{align*}
		|J_1| & \leq 20   (1 + z^2) \tdelta_1(z) \bigl( 1 - \Phi(z) \bigr) , &
        |J_3| & \leq 20   (1 + z) \tdelta_3(z) \bigl( 1 - \Phi(z) \bigr) .
    \end{align*}
    This completes the proof together with \cref{eq:j123}. 
\end{proof}

\subsection{A moment generating function bound}%
\label{sub:4.2}
The following proposition provides a moment generating function bound.  
\begin{proposition}
	\label{prop:5.3}
	Under the assumptions in \cref{thm:2.1}. For $0 \leq t \leq \tau$, 
	we have 
	\begin{equ}
		\E e^{tW} \leq \exp \biggl\{  \frac{t^2}{2} \bigl( 1 + \delta_1 (t)  + \delta_2 (t) \bigr) + \delta_3(t) t \biggr\} \quad \text{for $0 \leq t \leq \tau$}.
        \label{eq:p4.4-aa}
	\end{equ}
\end{proposition}
\begin{proof}
	[Proof of \cref{prop:5.3}]
	The proof is based on a standard technique of proving exponential bounds in Stein's method.  
	Specifically, let $h (t) = \E e^{tW}$. In order to bound $h(t)$, we need to find an upper bound for $h'(t)$ using Stein's method, and then obtain a bound for $(\log h(t))'$. This technique was firstly considered by \cite{Ch07S}, who proved a concentration inequality for
	exchangeable pairs. \cite{Ch13S,Sh18C,Fa19S} also used this technique to prove Cram\'er-type moderate deviation theorems. 

	By \cref{c:a0,c:a3}, 
	\begin{math}
		\E \{ |D| e^{tW} \} < \infty\text{ and }\E \{ |R| e^{tW} \} < \infty. 
	\end{math}
	Since $\E \{ D \vert X \} = \lambda(W + R)$, we have $\E \{ |W| e^{tW} \} < \infty$. 
	Applying \cref{eq:ewfw} with $f(w) = e^{tw}$, and by \cref{eq:ewfw}, we have   
	\begin{align*}
		h'(t) & = \E \bigl\{ W e^{tW} \bigr\} = \frac{t}{2\lambda} \E \biggl\{ D \int_{-\Delta}^{0} e^{ t (W + u) } \dd u \biggr\} - \E \{ R e^{tW} \}\\
			  & \leq t \E e^{tW} + t \E \biggl\{ \biggl\lvert \frac{1}{2\lambda} \conep{ D \Delta }{X} - 1 \biggr\rvert e^{tW}  \biggr\}  + \frac{t}{2\lambda} \biggl\lvert \E \biggl\{ D \int_{-\Delta}^0 \bigl(e^{t (W + u)} - e^{tW}\bigr) \dd u \biggr\} \biggr\rvert \\
			  & \quad + \E \{ |R| e^{tW} \} \\
			  & \leq t \E e^{tW} + t \E \biggl\{ \biggl\lvert \frac{1}{2\lambda} \conep{ D \Delta }{X} - 1 \biggr\rvert e^{tW}  \biggr\} + t \E \biggl\{ \biggl\lvert \frac{1}{2\lambda} \conep{ D^* \Delta }{X} \biggr\rvert e^{tW}  \biggr\} + \E \{ |R| e^{tW} \}, 
	\end{align*}
	where we used \cref{lem:psi} in the last line. 
	By \cref{c:a1,c:a2,c:a3}, we have for $0 \leq t \leq \tau$, 
	\begin{align*}
		h'(t) & \leq \bigl\{ t [1 + \delta_1(t) + \delta_2(t) ] + \delta_3(t) \bigr\} h(t), 
	\end{align*}
	which is 
	\begin{equ}
		(\log h(t))' \leq  t [1 + \delta_1(t) + \delta_2(t) ] + \delta_3(t).
        \label{eq:p4.4-01}
	\end{equ}
	Note that $h(0) = 1$ and that $\delta_1(t), \delta_2(t),\delta_3(t)$ are increasing, we have 
	\begin{align*}
		\log h(t) \leq \int_0^t \{ u [1 + \delta_1(u) + \delta_2(u) ] + \delta_3(u) \} \dd u
		\leq \frac{t^2}{2} [ 1 + \delta_1(t) + \delta_2(t) ] + t \delta_3(t), 
	\end{align*}
	which further implies \cref{eq:p4.4-aa}. 
\end{proof}

\subsection{Proof of \texorpdfstring{\cref{thm:2.1}}{Theorem 2.1}}%
\label{sub:4.3}

With the help of \cref{thm:exch0} and \cref{prop:5.3}, we are ready to give the proof of \cref{thm:2.1}. 
\begin{proof}
	[Proof of \cref{thm:2.1}]
	By \cref{prop:5.3}, we have  
	\begin{math}
		\E e^{tW} \leq e^{\theta} e^{t^2/2}
	\end{math}
	for $0 \leq t \leq \tau_0(\theta)$. 
	By \cref{c:a1,c:a2,c:a3}, we have \cref{c:c1,c:c2,c:c3} are satisfied with $\tau_0 = \tau_0(\theta)$, and  
	\begin{align*}
		\kappa_1(t) = e^{\theta} \delta_1(t), \quad 
		\kappa_2(t) = e^{\theta} \delta_2(t), \quad 
		\kappa_3(t) = e^{\theta} \delta_3(t). 
	\end{align*} 
	This proves \cref{thm:2.1} by applying \cref{thm:exch0}.
\end{proof}


\newcommand{\IndI}[1]{ #1^{({i})} }
\newcommand{\IndII}[1]{ #1^{({i}')} }
\newcommand{\II}{ {{i}'} } 
\newcommand{\JJ}{ {j}' } 

\section{Proof of other results}%
\label{sec:5}
\subsection{Proof of \texorpdfstring{\cref{thm:local}}{Theorem 3.1}}%
\label{sub:5.1}

We apply \cref{thm:exch0} to prove the result. In \cref{subb:5.1}, we construct an exchangeable pair $(X, X')$, and give several exponential moment inequalities in \cref{prop:5.1,prop:5.2} (see below). The main
proof of \cref{thm:local} is put in \cref{subb:5.2}, and we provide the proof of \cref{prop:5.2} in \cref{subb:5.3}. In what follows, we use $i, j , i', j', k, l$ to represent elements in $[n]$ and use $\alpha, \alpha'$ to represent elements in $\mathcal{J}$.
Moreover, throughout this section, we denote by $C$ an absolute constant, which may take different values in different places. 

\subsubsection{Exchangeable pairs and exponential moment inequalities}%
\label{subb:5.1}

To begin with, we construct an exchangeable pair for $X$. Let $X^* = \{ X_{\alpha}^* : \alpha \in \mathcal{J} \}$ be an independent copy of $X$. For each $J \subset \mathcal{J}$, let $X^{J} = \{ X_{\alpha}^{J} : \alpha \in \mathcal{J} \}$, where 
\begin{align*}
	X_{\alpha}^{J} = 
	\begin{cases}
		X_{\alpha}^* & \text{if $\alpha \in J$}, \\
		X_{\alpha}  & \text{if $\alpha \not\in J$}. \\
	\end{cases}
\end{align*}
For each $1 \leq i \leq n$, define the random vector $(\xi_1^{(i)}, \dots, \xi_n^{(i)})$ by 
\begin{math}
	\xi_j^{(i)} 
	= f_j(X_{J_j}^{J_i}). 
\end{math}
Let $I$ be a random variable uniformly chosen from $\{1, \dots, n\}$, which is independent of all others. Let $X' = X^{J_I}$ and $W' = \sum_{j = 1}^n \xi_j^{(I)}$. Then, it follows that $( X, X' )$ is exchangeable, and therefore so is $(W, W')$. 
Let $A_i = \{ j : J_j \cap J_i \neq \emptyset \}$. Then, for any $i$, we have $\xi_j = \xi_j^{(i)}$ for all $j \not\in A_i$.
Define 
\begin{equation}
    \label{eq-local-dd}
	D \coloneqq \xi_I - \xi_I^{(I)}, \quad 
	D^*  \coloneqq |\xi_I - \xi_I^{(I)}|, \quad
	\Delta \coloneqq W - W' = \sum_{j \in A_{I}} (\xi_j - \xi_j^{(I)}).
\end{equation}
Note that $\xi_i^{(i)}$ is independent of $X$ and that $\E \xi_i = 0$, then it follows that 
\begin{align*}
	\conep{D}{X} & = \frac{1}{n} \sum_{i = 1}^n \conep{ \xi_i - \xi_i^{(i)} }{X}
				 = \frac{1}{n} \sum_{i = 1}^n (\xi_i - \E \xi_i) = \frac{1}{n} W. 
\end{align*}
Therefore, the condition (D1) holds with $\lambda = 1/n$ and $R = 0$.
Applying \cref{eq:ewfw} with $f(w) = w$, we have $\E \{D \Delta\}/(2\lambda) =  \E W^2 = 1$.
Moreover, 
\begin{equ}
	\frac{1}{2\lambda}\conep{ D \Delta }{X, X^*} - 1 & = \frac{1}{2\lambda} \bigl[ \conep{ D \Delta }{X, X^*} - \E \{ D \Delta \} \bigr] \\
														   & = \frac{1}{2} \sum_{i = 1}^n \bigl[ ( \xi_i - \xi_i^{(i)} )(\xi_j - \xi_j^{(i)}) - \E \{ ( \xi_i - \xi_i^{(i)} )(\xi_j - \xi_j^{(i)}) \}\bigr] , \\
	\frac{1}{\lambda}\conep{ D^* \Delta }{X, X^*} & = \sum_{i = 1}^n \lvert \xi_i - \xi_i^{(i)} \rvert(\xi_j - \xi_j^{(i)}).
    \label{eq-local-d2}
\end{equ}
To verify \cref{c:a0,c:a1,c:a2,c:a3}, we need to prove the following proposition. The first proposition is important in verifying \cref{c:a0}.
\begin{proposition}
	\label{prop:5.1}
	Assume that \cref{eq:c31,eq:c32,eq:c33} hold, we have $\E e^{tW} < \infty$ for $0 \leq t \leq a / d$. 
\end{proposition}
The proof of \cref{prop:5.1} is based on \cref{eq:c31,eq:c32} and independence, and the details is given in the Supplementary Material \citep{Zh21S}. The following proposition plays an important role in verifying \cref{c:a1,c:a2,c:a3}.
\begin{proposition}
	\label{prop:5.2}
	For each $i, j \in [n]$, let $\eta_{ij} = (\xi_i - \xi_i^{(i)}) ( \xi_j - \xi_j^{(i)} ) - \E \{ (\xi_i - \xi_i^{(i)}) ( \xi_j - \xi_j^{(i)} ) \}$ or $\eta_{ij} = \lvert \xi_{i} - \xi_i^{(i)} \rvert (\xi_j - \xi_j^{(i)})$. Under the assumptions of \cref{thm:local}, we have for $0 \leq
	t \leq a/(8dsb^2)$, 
	\begin{equ}
		\E \biggl\{ \biggl( \sum_{i \in [n]} \sum_{j \in A_i} \eta_{ij} \biggr)^2 e^{tW}\biggr\}
		& \leq C   n a^{-4} d^3 s^7 b (1  + n a^{-2} d s^3 b t^2)\E e^{tW}.
        \label{eq-p5.2-aa}
	\end{equ}
\end{proposition}
The proof of \cref{prop:5.2} is put in \cref{subb:5.3}.
\subsubsection{Proof of \tcref{thm:local}}%
\label{subb:5.2}

Now, we are ready to prove \cref{thm:local} with the help of \cref{prop:5.1,prop:5.2}.
\begin{proof}
	[Proof of \cref{thm:local}]
	Recall that condition (D1) is satisfied with $\lambda = 1/n$ and $R = 0$. 
	As $D = \xi_I - \xi_I^{(I)}$, by Cauchy's inequality and recalling the assumption that $\sum_{i = 1}^n \E \xi_i^2 = 1$, it follows that 
	\begin{align*}
		\E \{ |D| e^{tW} \} \leq ( \E \lvert \xi_I - \xi_I^{(I)} \rvert^2 ) ( \E e^{2 t W} )^{1/2} \leq 2 (\E e^{2tW})^{1/2}. 
	\end{align*}
	By \cref{prop:5.1}, and recalling that $s \geq 1, d \geq 1, b \geq 1$, we have \cref{c:a0} is satisfied for $0 \leq t \leq a / (8 d s b^2)$.  

	Moreover, by Jensen's inequality and H\"older's inequality, we have 
	\begin{align*}
		\E\bigl\{\bigl\lvert 1 - \frac{1}{2\lambda} \conep{D\Delta}{\bX}\bigr\rvert \ex{tW}\bigr\} 
		& \leq \E\bigl\{\bigl\lvert 1 - \frac{1}{2\lambda} \conep{D\Delta}{\bX, X^*}\bigr\rvert \ex{tW}\bigr\}, \\
		& \leq (\E e^{tW})^{1/2} \biggl( \E\bigl\{\bigl\lvert 1 - \frac{1}{2\lambda} \conep{D\Delta}{\bX, X^*}\bigr\rvert^2 \ex{tW}\bigr\} \biggr)^{1/2}.
	\end{align*}
	Let $\tau_0 = \min\{a /
	(8 d s b^2), \delta^{-1/3}\}$. 
	By \cref{eq-local-d2,prop:5.2}, we have \cref{c:a1} in \cref{thm:2.1} are satisfied with $\delta_1(t) = C \delta (1 + t)$. Similarly, \cref{c:a2} is also satisfied with $\delta_2(t) = C \delta(1 + t)$. Applying \cref{thm:2.1} yields the desired result. 
\end{proof}

\subsubsection{Proof of \tcref{prop:5.2}}
\label{subb:5.3}

Before giving the proof of \cref{prop:5.2}, we need to prove some preliminary lemmas, which provide moment inequalities for $g_{i\alpha}(X_{\alpha})$ and $\xi_i$.  
The proofs are given in the Supplementary Material \citep{Zh21S}. 

\begin{lemma}
	\label{lem:5.3}
	We have for all $i \in [n]$ and $\alpha \in J_i$, 
	\begin{align}
		\E \lvert g_{i \alpha}(X_{\alpha}) \rvert^6 & \leq 120 a^{-6} b, \label{eq:l5.3-aa}\\
		\E \lvert \xi_i \rvert^6 & \leq 120 a^{-6} s^6 b. \label{eq:l5.3-ab}
	\end{align}
\end{lemma}
\begin{lemma}
	\label{lem:5.4}	
	For any $i \in [n]$ and $\alpha \in J_i$, we have for $0 \leq t \leq a/(4 sd b^2)$,  
	\begin{align*}
		\E \bigl\{ g_{i \alpha}(X_{\alpha})^6 e^{t W} \bigr\} & \leq C a^{-6} b^2 \E e^{tW}.
	\end{align*}
\end{lemma}

\begin{proof}[Proof of \cref{prop:5.2}]
In this proof, $C$ denotes  an  absolute positive constant, and $O(a)$ denotes a quantity satisfying that $\lvert O(a) \rvert \leq Ca$. 
	Expanding the square term in the left hand side of \cref{eq-p5.2-aa} yields 
	\begin{equ}
		\text{LHS of \cref{eq-p5.2-aa}} 
		& = \sum_{i \in [n]} \sum_{j \in A_i} \sum_{i' \in [n]}  \sum_{j' \in A_i} \E \bigl\{ \eta_{ij} \eta_{i'j'} e^{tW} \bigr\}.
        \label{eq-p5.2-01}
	\end{equ}
	In what follows, we focus on estimating each summand on the right hand of \cref{eq-p5.2-01}, which is based on independence, Taylor's expansion and exponential moment bounds. 

	In order to calculate the summand on the right hand side of \cref{eq-p5.2-01}, we need to introduce some notation. Let $\tilde{X} = (\tilde{X}_m : m \in \mathcal{J})$ be an independent copy of $X$, which is also independent of $X^*$. Again, for any $J \subset \mathcal{J}$,
	let $\tilde{X}^{J} = (\tilde{X}_m^J : m \in \mathcal{J})$ be defined as 
	\begin{align*}
		\tilde{X}_m^J 
		= 
		\begin{cases}
			\tilde{X}_m & \text{if $m \in J$}, \\
			X_m & \text{if $m \not\in J$}.
		\end{cases}
	\end{align*}
	Let $J_{iji'j'} = J_i \cup J_j \cup J_{i'} \cup J_{j'}$. For each $1 \leq k \leq n$, let $\tilde{\xi}_k^{\,(i,j,i',j')} = f_k (\tilde{X}_{J_k}^{J_{iji'j'}})$ and let $\tilde{W}^{(i,j,i',j')} = \sum_{k \in [n]} \tilde{\xi}_k^{\,(i,j,i',j')}$. Then, we have $\tilde{W}^{(i,j,i',j')}$ is independent of $(\eta_{ij}, 
	\eta_{i',j'})$ and has the same distribution as $W$. 
	For $1 \leq i \leq n$, let $A_i = \{ j : J_i \cap J_j \neq \emptyset \}$.  From \cref{eq:c33}, 
	\begin{equation}
		\label{eq:Aibound}
		\lvert A_i \rvert \leq \min\{ sd, n \} . 
	\end{equation}
	For any $i, j, i', j'$, let $A_{iji'j'} = A_i \cup A_j \cup A_{i'} \cup A_{j'}$. 
	Moreover, it follows that 
	\begin{math}
		W - \tilde{W}^{(i,j,i',j')}	= \sum_{k \in A_{iji'j'}} ( \xi_k - \tilde{\xi}_k^{\,(i,j,i',j')} ).
	\end{math}
	Applying Taylor's expansion to right hand side of \cref{eq-p5.2-01}, we obtain
	\begin{align*}
	\ML \text{RHS of \cref{eq-p5.2-01}}\\
		& = \sum_{i \in [n]} \sum_{j \in A_i} \sum_{i' \in [n]} \sum_{j' \in A_{i'}} \E \bigl\{ \eta_{ij} \eta_{i'j'} e^{t \tilde{W}^{(i,j,i',j')}} \bigr\} \\
		& \quad + t \sum_{i \in [n]} \sum_{j \in A_i} \sum_{i' \in [n]} \sum_{j' \in A_{i'}} \sum_{k \in A_{iji'j'} } \E \bigl\{ \eta_{ij} \eta_{i'j'} \xi_k e^{t \tilde{W}^{( i,j,i',j' )}} \bigr\}\\
		& \quad - t \sum_{i \in [n]} \sum_{j \in A_i} \sum_{i' \in [n]} \sum_{j' \in A_{i'}} \sum_{k \in A_{iji'j'} } \E \bigl\{ \eta_{ij} \eta_{i'j'} \tilde{\xi}_k^{\,(i,j,i',j')} e^{t \tilde{W}^{( i,j,i',j' )}} \bigr\}\\
		& \quad + O(t^2) \sum_{i \in [n]} \sum_{j \in A_i} \sum_{i' \in [n]} \sum_{j' \in A_{i'}} \E \biggl\{ \lvert \eta_{ij} \eta_{i'j'} \rvert \biggl( \sum_{k \in A_{iji'j'}} ( \xi_{k} - \tilde{\xi}_{k}^{\,(i,j,i',j')} ) \biggr)^2 e^{t \tilde{W}^{(i,j,i',j')}} \biggr\}\\
		& \quad + O(t^2) \sum_{i \in [n]} \sum_{j \in A_i} \sum_{i' \in [n]} \sum_{j' \in A_{i'}} \E \biggl\{ \lvert \eta_{ij} \eta_{i'j'} \rvert \biggl( \sum_{k \in A_{iji'j'}} ( \xi_{k} - \tilde{\xi}_{k}^{\,(i,j,i',j')} ) \biggr)^2 e^{t W} \biggr\}\\
		& \eqqcolon H_1 + H_2 - H_3 + H_4 + H_5. 
	\end{align*}
	For $H_1$, recalling that $\tilde{W}^{(i,j,i',j')}$ is independent of $(\eta_{ij}, \eta_{i',j'})$ and has the same distribution as $W$, we have 
	\begin{equ}
		\E \{ \eta_{ij} \eta_{i'j'}  e^{t \tilde{W}^{(i,j,i',j')}} \}  = \E \{ \eta_{ij}  \eta_{i'j'}\} \E e^{tW}.
        \label{eq-p5.2-H1}
	\end{equ}
	Now, note that $\eta_{ij}$ is independent of $\eta_{i'j'}$ when $i' \in A_{ij}^c$ and $j' \in A_{ij}^c$, and that $\E \eta_{ij} = \E \eta_{i'j'} = 0$, thus, 
	\begin{align}
		\E \{ \eta_{ij} \eta_{i'j'} \} = 0 \quad \text{if $i' \in A_{ij}^c$ and $j' \in A_{ij}^c$}.
		\label{eq:eta0}
	\end{align}
Moreover, by \cref{lem:5.3} and H\"older's inequality, 
\begin{equ}
	\E \lvert \eta_{ij} \rvert^2 \leq C ( \E \xi_i^4 + \E \xi_j^4 )  \leq C a^{-4}  s^4 b^{2/3}. 
    \label{eq:eta2}
\end{equ}
By \cref{eq:Aibound} and \cref{eq:c33},
\begin{equ}
	\lvert \{ i' : j' \in A_{i'}\} \rvert = |\{ i' : J_{i'} \cap J_{j'} \neq \emptyset \}| = \lvert A_{j'} \rvert \leq \min \{ sd, n \}, 
    \label{eq:Aibound2}
\end{equ}
and then we obtain  
	\begin{equ}
		\lvert H_1 \rvert
		& \leq \E e^{tW}\sum_{i \in [n]} \sum_{j \in A_i} \sum_{i' \in [n]} \sum_{j' \in A_{i'}} \E \lvert \eta_{ij} \eta_{i'j'} \rvert \mathbf{1}_{\{ i' \in A_{ij} \text{ or } j' \in A_{ij} \}}\\
		& \leq C a^{-4} s^4 b^{2/3} \E e^{tW}\sum_{i \in [n]} \sum_{j \in A_i} \sum_{i' \in [n]} \sum_{j' \in A_{i'}}   \mathbf{1}_{\{ i' \in A_{ij} \text{ or } j' \in A_{ij} \}}\\
		& \leq C a^{-4} s^4 b^{2/3} \E e^{tW} \biggl( \sum_{i \in [n]} \sum_{j \in A_i} \sum_{i' \in A_{ij}} \sum_{j' \in A_{i'}}   1 + \sum_{i \in [n]} \sum_{j \in A_i} \sum_{j' \in A_{ij}} \sum_{i' : j' \in A_{i'}}   1 \biggr)  \\
		& \leq C a^{-4} s^4 b^{2/3} ( n d^3 s^3 ) \E e^{tW} \leq C n a^{-4} d^3 s^7 b^{2/3} \E e^{tW}.
        \label{eq-p5.2-M1}
	\end{equ}
	To calculate $H_2$, we need to introduce some more notation. For fixed $i,j,i',j',k$, let $J_{iji'j'k} =  J_{iji'j'} \cup J_k$, and for each $1 \leq l \leq n$, let $\tilde{\xi}_l^{\,(i,j,i',j',k)} = f_l(\tilde{X}_{J_l}^{J_{iji'j'k}})$ and $\tilde{W}^{(i,j,i',j',k)} =
	\sum_{l \in [n]} \tilde{\xi}_l^{\,(i,j,i',j',k)}$. Then, we have $\tilde{W}^{(i,j,i',j',k)}$ is independent of $(\eta_{ij}, \eta_{i',j'}, \xi_k)$ and has the same distribution as $W$, and thus, 
\begin{equ}
	\E \bigl\{ \eta_{ij} \eta_{i'j'} \xi_k e^{t \tilde{W}^{(i,j,i',j',k)}} \bigr\} = \E \{ \eta_{ij} \eta_{i'j'} \xi_k e^{t \tilde{W}^{(i,j,i',j',k)}} \} =	\E \{ \eta_{ij} \eta_{i'j'} \xi_k	 \} \E e^{t W}.
    \label{eq-p5.2-02}
\end{equ}
	Moreover, note that 
	\begin{equ}
		\tilde{W}^{(i,j,i',j')} - \tilde{W}^{(i,j,i',j',k)}
		= \sum_{l \in A_k} ( \tilde{\xi}_l^{\,(i,j,i',j')} - \tilde{\xi}_{l}^{\,(i,j,i',j',k)} ). 
        \label{eq-p5.2-03}
	\end{equ}
	By Taylor's expansion again and by \cref{eq-p5.2-03}, we have
	\begin{align*}
		\begin{split}
			H_2 & =  t \sum_{i \in [n]} \sum_{j \in A_i} \sum_{i' \in [n]} \sum_{j' \in A_{i'}} \sum_{k \in A_{iji'j'}} \E \bigl\{ \eta_{ij} \eta_{i'j'} \xi_k e^{t \tilde{W}^{(i,j,i',j',k)}} \bigr\} \\
				& \quad + O(t^2) \sum_{i \in [n]} \sum_{j \in A_i} \sum_{i' \in [n]} \sum_{j' \in A_{i'}} \sum_{k \in A_{iji'j'}}\sum_{l \in A_k} \E \bigl\{ \lvert \eta_{ij} \eta_{i'j'} \xi_k (\tilde\xi_l^{\,(i,j,i',j')} - \tilde{\xi}_l^{\,(i,j,i',j'k)}) \rvert e^{t \tilde{W}^{(i,j,i',j')}} \bigr\}\\
				& \quad + O(t^2) \sum_{i \in [n]} \sum_{j \in A_i} \sum_{i' \in [n]} \sum_{j' \in A_{i'}} \sum_{k \in A_{iji'j'}}\sum_{l \in A_k} \E \bigl\{ \lvert \eta_{ij} \eta_{i'j'} \xi_k (\tilde\xi_l^{\,(i,j,i',j')} - \tilde{\xi}_l^{\,(i,j,i',j'k)}) \rvert e^{t \tilde{W}^{(i,j,i',j',k)}} \bigr\}\\
				& \eqqcolon H_{21} + H_{22}  + H_{23}. 
		\end{split}
	\end{align*}
	For $H_{21}$, by \cref{eq-p5.2-02} and by the fact that $A_{iji'j'} = A_{ij} \cup A_{i'j'}$, we have 
	\begin{equ}
		|H_{21}| & \leq t \E e^{tW} \sum_{i \in [n]} \sum_{j \in A_i} \sum_{i' \in [n]} \sum_{j' \in A_{i'}} \sum_{k \in A_{iji'j'}} |\E \{ \eta_{ij} \eta_{i'j'} \xi_k \}| \\
				 & \leq t \E e^{tW} \sum_{i \in [n]} \sum_{j \in A_i} \sum_{i' \in [n]} \sum_{j' \in A_{i'}} \biggl( \sum_{k \in A_{ij}} |\E \{ \eta_{ij} \eta_{i'j'} \xi_k\}| + \sum_{k \in A_{i'j'}} |\E \{ \eta_{ij} \eta_{i'j'} \xi_k\}|\biggr)  \\
		& = 2 t \E e^{tW} \sum_{i \in [n]} \sum_{j \in A_i} \sum_{i' \in [n]} \sum_{j' \in A_{i'}} \sum_{k \in A_{ij}} |\E \{ \eta_{ij} \eta_{i'j'}
		\xi_k\}|, 
        \label{eq-p5.2-H21}
	\end{equ}
	where the last line is based on symmetry. If $i', j' \in A_{ijk}^c$, then $\eta_{i'j'}$ is independent of $(\eta_{ij}, \xi_k)$, and thus 
	\begin{math}
		\E \{ \eta_{ij} \eta_{i'j'}
		\xi_k\} = \E \{ \eta_{ij} \xi_k	\} \E \eta_{i'j'} = 0. 
	\end{math}
	Moreover, by H\"older's inequality and \cref{lem:5.3}, 
	\begin{align*}
		\E \lvert \eta_{ij} \eta_{i'j'} \xi_k \rvert \leq C ( \E |\xi_i|^5 + \E \lvert \xi_j \rvert^5 + \E \lvert \xi_{i'} \rvert^5 + \E \lvert \xi_{j'} \rvert^5 + \E \lvert \xi_k \rvert^5 ) \leq C a^{-5} s^5 b^{5/6}.
	\end{align*}
	Therefore, by \cref{eq:Aibound,eq:Aibound2},
	\begin{equ}
		\lvert H_{21} \rvert
		& \leq C t a^{-5} s^5 b^{5/6} \E e^{tW} \sum_{i \in [n]} \sum_{j \in A_i} \sum_{i' \in [n]} \sum_{j' \in A_{i'}} \sum_{k \in A_{ij}} \mathbf{1}_{ \{ i' \in A_{ijk} \text{ or } j' \in A_{ijk} \} }\\
		& \leq C t a^{-5} s^5 b^{5/6} \E e^{tW} \sum_{i \in [n]} \sum_{j \in A_i}  \sum_{k \in A_{ij}} \biggl( \sum_{i' \in A_{ijk}} \sum_{j' \in A_{i'}} 1 + \sum_{j' \in A_{ijk}} \sum_{i' : j' \in A_{i'}} 1 \biggr)\\
		& \leq C t a^{-5} s^5 b^{5/6} \bigl( n \min \{ n, sd \}^4 \bigr)  \E e^{tW}  \\
		& \leq C n  a^{-4} d^3 s^7 b (1 + t^2 n a^{-2} d s^3 b) \E e^{tW},  
        \label{eq-p5.2-M2}
	\end{equ}
	where we used Cauchy's inequality and the fact that $b \geq 1$ in the last line. 

	For $H_{22}$, by Young's inequality, we have 
	\begin{equ}
		\MoveEqLeft \E \Bigl\{ \Bigl\lvert \eta_{ij} \eta_{i'j'} \xi_k (\tilde\xi_l^{\,(i,j,i',j')} - \tilde{\xi}_l^{\,(i,j,i',j'k)}) \Bigr\rvert e^{t \tilde{W}^{(i,j,i',j')}} \Bigr\}\\
		& \leq \E \Bigl\{ \Bigl\lvert \eta_{ij} \eta_{i'j'} \xi_k \tilde\xi_l^{\,(i,j,i',j')}  \Bigr\rvert e^{t \tilde{W}^{(i,j,i',j')}} \Bigr\} + \E \Bigl\{ \Bigl\lvert \eta_{ij} \eta_{i'j'} \xi_k  \tilde{\xi}_l^{\,(i,j,i',j'k)} \Bigr\rvert e^{t \tilde{W}^{(i,j,i',j')}}
		\Bigr\}\\
		& \leq \frac{2}{3} \E \bigl\{ \lvert \eta_{ij} \rvert^3 e^{t \tilde{W}^{(i,j,i',j')}} \bigr\} + \frac{2}{3} \E \bigl\{ \lvert \eta_{i'j'} \rvert^3 e^{t \tilde{W}^{(i,j,i',j')}} \bigr\} + \frac{1}{3} \E \bigl\{ \xi_k^6 e^{t \tilde{W}^{(i,
		j,i',j')}} \bigr\} \\
		& \quad + \frac{1}{6} \E \bigl\{ \lvert \tilde\xi_l^{\,(i,j,i',j')} \rvert^6 e^{t \tilde{W}^{(i,j,i',j')}} \bigr\} + \frac{1}{6} \E \bigl\{ \lvert \tilde\xi_l^{\,(i,j,i',j',k)} \rvert^6 e^{t \tilde{W}^{(i,j,i',j')}} \bigr\}.
        \label{eq-p5.2-04}
	\end{equ}
	For the first two terms of the right hand side of \cref{eq-p5.2-04}, recalling that $\tilde{W}^{(i,j,i',j')}$ is independent of $(\eta_{i,j}, \eta_{i',j'})$ and has the same distribution as $W$, and by the definition of $\eta_{ij}$ and by \cref{lem:5.3}, we have 
	\begin{equ}
		\E \bigl\{ \lvert \eta_{ij} \rvert^3 e^{t \tilde{W}^{(i,j,i',j')}} \bigr\} 
		& =  \E \lvert \eta_{ij} \rvert^3 \E e^{tW} \leq C \E e^{tW}( \E \lvert \xi_i \rvert^6 + \E \lvert \xi_j \rvert^6 )\leq C a^{-6} s^6 b \E e^{tW}, \\
		\E \bigl\{ \lvert \eta_{i'j'} \rvert^3 e^{t \tilde{W}^{(i,j,i',j')}} \bigr\} 
		& =  \E \lvert \eta_{i'j'} \rvert^3 \E e^{tW} \leq C \E e^{tW}( \E \lvert \xi_{i'} \rvert^6 + \E \lvert \xi_{j'} \rvert^6 ) \leq C a^{-6} s^6 b \E e^{tW}. 
        \label{eq-p5.2-09}
	\end{equ}
	For the third term of the right hand side of \cref{eq-p5.2-04}, by \cref{eq:c31},
	\begin{equ}
		\E \bigl\{ \xi_k^6 e^{t \tilde{W}^{(i,j,i',j')}} \bigr\}
		& \leq s^5 \sum_{m \in J_k} \E \bigl\{ g_{km}(X_m)^6 e^{t \tilde{W}^{(i,j,i',j')}} \bigr\}. 
        \label{eq-p5.2-05}
	\end{equ}
	Now, if $m \in J_k \setminus J_{iji'j'}$, then $X_m$ is independent of $\tilde{W}^{(i,j,i',j')}$, and thus 
	\begin{equ}
		\E \bigl\{ g_{km}(X_m)^6 e^{t \tilde{W}^{(i,j,i',j')}} \bigr\}
		& = \E g_{km}(X_m)^6  \E e^{t W}
		\leq C a^{-6} b \E e^{tW}, 
        \label{eq-p5.2-06}
	\end{equ}
	where we used \cref{lem:5.3} in the last inequality. If $m \in J_k \cap J_{iji'j'}$, then $(X_m, \tilde{W}^{(i,j,i',j')})$ has the same distribution as $(X_m, W)$ by the construction of $\tilde{W}^{(i,j,i',j')}$, and thus 
	\begin{equ}
		\E \bigl\{ g_{km}(X_m)^6 e^{t \tilde{W}^{(i,j,i',j')}} \bigr\}
		& = \E \bigl\{ g_{km}(X_m)^6 e^{t W} \bigr\} \leq C a^{-6} b^2 \E e^{tW},
        \label{eq-p5.2-07}
	\end{equ}
	where we used \cref{lem:5.4} in the last inequality. By \cref{eq-p5.2-06,eq-p5.2-07} and recalling that $b \geq 1$, we have 
	\cref{eq-p5.2-05} can be further bounded by 
	\begin{equ}
		\E \bigl\{ \xi_k^6 e^{t \tilde{W}^{(i,j,i',j')}} \bigr\} \leq C a^{-6} s^6 b^2 \E e^{tW}.
        \label{eq-p5.2-08}
	\end{equ}
	Using a similar argument to \cref{eq-p5.2-08}, we have 
	\begin{equ}
		 \E \bigl\{ |\tilde\xi_k^{\,(i,j,i',j')}|^6 e^{t \tilde{W}^{(i,j,i',j')}} \bigr\} \leq C a^{-6} s^6 b^2 \E e^{tW}, \\
		 \E \bigl\{ |\tilde\xi_k^{\,(i,j,i',j',k)}|^6 e^{t \tilde{W}^{(i,j,i',j')}} \bigr\} \leq C a^{-6} s^6 b^2 \E e^{tW}.
        \label{eq-p5.2-10}
	\end{equ}
	Combining \cref{eq-p5.2-04,eq-p5.2-09,eq-p5.2-08,eq-p5.2-10}, and by \cref{eq:Aibound},  we have 
	\begin{equ}
		\lvert H_{22} \rvert & \leq C t^2 a^{-6} s^{6} b^2  ( n^2 d^4 s^4 ) \E e^{tW} \leq Ct^2 n^2 a^{-6} d^4 s^{10} b^2 \E e^{tW}. 
        \label{eq-p5.2-H22}
	\end{equ}
	Similarly, 
	\begin{equ}
		|H_{23}| & \leq C  t^2 n^2 a^{-6} d^4 s^{10} b^2 \E e^{tW}.
        \label{eq-p5.2-H23}
	\end{equ}
	Combining \cref{eq-p5.2-H21,eq-p5.2-H22,eq-p5.2-H23}, we have 
	\begin{equ}
		|H_2| & \leq C t n a^{-5} d^4 s^{9} b^{5/6} \E e^{tW} + C t^2 n^2 a^{-6} d^4 s^{10} b^2 \E e^{tW}
        \label{eq-p5.2-H2}
	\end{equ}
	For $H_3$, note that $(\tilde{\xi}_k^{\,(i,j,i',j')}, \tilde{W}^{(i,j,i',j')})$ is independent of $(\eta_{ij}, \eta_{i',j'})$ and has the same distribution as $(\xi_k, W)$. By \cref{eq:eta0,eq:eta2}, we have 
	\begin{equ}
		|H_3| & \leq t \sum_{i \in [n]} \sum_{j \in A_i} \sum_{i' \in [n]} \sum_{j' \in A_{i'}} \sum_{k \in A_{iji'j'}} \E \lvert \eta_{ij} \eta_{i'j'} \rvert \E \bigl\{ \xi_k e^{tW} \bigr\} \mathbf{1}_{ \{ i' \in A_{ij} \text{ or } j' \in A_{ij} \} }
		\\
			  & \leq C t a^{-4}  s^4 b^{2/3} \sum_{i \in [n]} \sum_{j \in A_i} \sum_{i' \in [n]} \sum_{j' \in A_{i'}} \sum_{k \in A_{iji'j'}} \E \bigl\{ \lvert \xi_k \rvert e^{tW} \bigr\} \mathbf{1}_{ \{ i' \in A_{ij} \text{ or } j' \in A_{ij} \} }
			  \\
			  & \leq C t n a^{-4} s^4 b^{2/3} \min \{ n, sd \}^4 \max_{k \in [n]} \E \bigl\{ \lvert \xi_k \rvert e^{tW} \bigr\} .
        \label{eq-p5.2-H3}
	\end{equ}
	By H\"older's inequality, \cref{eq:c31}, \cref{eq:c33} and \cref{lem:5.4}, for all $k \in [n]$, 
	\begin{equ}
		\E \bigl\{ \lvert \xi_k \rvert e^{tW} \bigr\} 
		& \leq \sum_{m \in J_k} \E \bigl\{ g_{km}(X_m) e^{tW} \bigr\}\\
		& \leq \sum_{m \in J_k} \biggl( \E \bigl\{ g_{km}(X_m)^6 e^{tW} \bigr\} \biggr)^{1/6} ( \E e^{tW} )^{5/6} \\
		& \leq C a^{-1} s b^{1/3} \E e^{tW}.
        \label{eq-p5.2-11}
	\end{equ}
	Substituting \cref{eq-p5.2-11} to \cref{eq-p5.2-H3} yields 
	\begin{align*}
		\lvert H_3 \rvert 
		& \leq C t n a^{-5}  s^{5} b \min \{ n, sd \}^4 \E e^{tW}  \leq C n  a^{-4} d^3 s^7 b (1 + t^2 n a^{-2} d s^3 b) \E e^{tW}. 
	\end{align*}

	For $H_4$ and $H_5$, similar to \cref{eq-p5.2-04}--\cref{eq-p5.2-H22}, we have 
	\begin{equ}
		H_4 + H_5 = O(t^2) n^2 a^{-6} d^4 s^{10} b^2 \E e^{tW}.
        \label{eq-p5.2-H4}
	\end{equ}
	Combining \cref{eq-p5.2-01,eq-p5.2-H1,eq-p5.2-H2,eq-p5.2-H3,eq-p5.2-H4} yields \cref{eq-p5.2-aa}.
	This completes the proof. 
\end{proof}

\subsection{Proof of \tcref{thm:cw}}
In this subsection, we denote by $C$ a general constant that depends only on $\beta$, where $0 < \beta < 1$.  
Let $\mathcal{X}$ be the sigma field generated by $(X_1, \dots, X_n)$. 
For each $1 \leq i \leq n$,  let 
$X_i'$ be conditionally independent of $X_i$  with the conditional distribution of $X_i$ given $\left\{ X_j, j \neq i \right\}$. 
Let $I$ be a random index that is uniformly distributed over $\left\{ 1, \dots, n \right\}$ 
and independent of all others. Define $S_n' = S_n - X_I + X_I'$; then $(S_n, S_n')$ is an exchangeable pair. 

For $n \leq 16 \max\{1, \beta / (1 - \beta)\}$, and for $0 \leq z \leq \sqrt{n}$, we have $z \leq z_{\beta} :=4  \sqrt{ \max\{1,\beta / (1 - \beta)\}}.$
By \cite[][Theorem 3.2]{Sh19B}, for $0 \leq z \leq z_\beta$, 
\begin{align*}
 \bigl\lvert \IP(W > z ) - (1 - \Phi(z)) \bigr\rvert & \leq C n^{-1/2} \leq C n^{-1/2} (1 - \Phi(z_{\beta})) \leq C n^{-1/2}(1 - \Phi(z)).
\end{align*}
Hence 
\cref{eq:tcw-a} holds. For $n > 16\max\{1,\beta(1 - \beta)\}$, we apply \cref{thm:exch0} to prove the moderate deviation result.  To this end, we need to prove the following propositions. 

\begin{proposition}
    \label{pro:cw1}
    Let $\bar X = S_n/n$ and let 
    \begin{equ}
		\label{eq:5.67}
        R_1 = \conep{S_n - S_n'}{\mathcal{X}} - (1 - \beta) \bar{X}.  
    \end{equ}
	For $n > 16 \max\{1, \beta / (1 - \beta)\}$ and $0 \leq t \leq \sqrt{n}$, we have  
    \begin{align}
        \label{eq:c5.3-b}
        \E \ex{t W} & \leq C \ex{t^2/2},  \\
		\label{eq:p1.1-b}
		\E \{ \lvert \bar X \rvert \ex{tW} \} & \leq C n^{-1/2} (1 + t) \ex{t^2/2}, \\
        \label{eq:p1.1-a} 
		\E \{|R_1| \ex{t W}\} & \leq C n^{-1} (1 + t^2) \ex{t^2/2}, \\
        \E \Bigl\{\bigl\lvert \conepbc{ (S_n - S_n')^2 }{\mathcal{X}} - 2 \bigr\rvert \ex{t W}\Bigr\} & \leq C n^{-1/2} (1 + t) \ex{t^2/2}, 
		\label{eq:p1.2-a} \\
        \E \Bigl\{\bigl\lvert \conepbc{ (S_n - S_n')^3 }{\mathcal{X}} \bigr\rvert \ex{t W}\Bigr\} & \leq C n^{-1/2}(1 + t) \ex{t^2/2}. 
        \label{eq:p1.2-b} 
    \end{align} 
\end{proposition}

With the help of \cref{pro:cw1}, we can check the \cref{c:c1,c:c2,c:c3} immediately. To see this, 
let $W' = n^{-1/2} (1 - \beta)^{1/2} S_n'$, $D = \Delta = W - W'$ and $D^* = n^{-1/2} (1 - \beta)^{1/2} + n^{1/2} (1 - \beta)^{-1/2} D^2$. Then, it follows that $|D| \leq D^*$ and $D^*$ is symmetric with respect to $W$ and $W'$. 
Observe that 
\begin{align*}
    \conep{D}{ \mathcal{X} } & = \conep{ W - W' }{\mathcal{X}} = n^{-1/2}(1 - \beta)^{1/2} \conep{S_n - S_n'}{ \mathcal{X}}  = \lambda (W + R), 
\end{align*}
where $\lambda = (1-\beta) /n $ and $R = n^{1/2} (1 - \beta)^{1/2} R_1$. Moreover, 
\begin{align*}
    \frac{1}{2\lambda} \conep{D \Delta}{\mathcal{X}} - 1 = \frac{1}{2\lambda} \conepbc{(W - W')^2}{ \mathcal{X}} - 1 = \frac{1}{2} \Bigl( \conepbc{ (S_n - S_n')^2 }{\mathcal{X}} - 2\Bigr), 
\end{align*}
and 
\begin{align*}
	\frac{1}{\lambda} \conep{ D^* \Delta }{\mathcal{X}} & = \frac{n^{-1/2}(1 - \beta)^{1/2}}{\lambda} \conep{ W - W' }{\mathcal{X}} + \frac{n^{1/2}(1 - \beta)^{-1/2}}{\lambda} \conep{ (W - W')^3 }{\mathcal{X}}\\
																		& =  \conep{S_n - S_n'}{ \mathcal{X} } + \conep{(S_n - S_n')^3}{\mathcal{X}}. 
\end{align*}
Hence, by \cref{pro:cw1}, we have that \cref{c:c1,c:c2,c:c3} are satisfied with $\tau_0 = n^{1/2}, \delta_1 (t) = \delta_2(t) = C n^{-1/2} (1 + t)$ and $\delta_3 (t) = C n^{-1}(1 + t^2)$. This completes the proof of \cref{thm:cw} by applying \cref{thm:exch0}. 

It suffices to prove \cref{pro:cw1}; 
to this end, we need to show some preliminary lemmas. 

\begin{lemma}
    \label{lem:c5.1} 
    For $0 \leq \theta < 1$, we have 
    \begin{align}
        \label{eq:c5.1-a}
        \E \ex{\theta \xi^2/2} & \leq C_{\theta}, 
    \end{align}
    where $C_{\theta} > 0$ is a constant depending on $\theta$.
\end{lemma}

\begin{lemma}
	For $1 \leq m \leq n$, let $T_m = \xi_1 + \dots + \xi_m$ and assume that $n \geq 16 \max \{ 1, \beta/(1 - \beta)\}$. 
    We have  for all $m \in [n]$ and $0 \leq t \leq \sqrt{n} $, 
    \begin{align}
        \label{eq:c5.2-a}
		\E \exp \biggl( \Bigl( \frac{\beta}{2n} + \frac{2\beta}{n^2} \Bigr) T_m^2 + \Bigl(\frac{1 - \beta}{n}\Bigr)^{1/2}t  T_m \biggr) & \leq C \ex{t^2/2} ,\\
        \label{eq:c5.2-b}
		\E \biggl\{T_m^2 \exp \biggl( \Bigl( \frac{\beta}{2n} + \frac{2\beta}{n^2} \Bigr) T_m^2 + \Bigl(\frac{1 - \beta}{n}\Bigr)^{1/2}t  T_m \biggr)\biggr\} & \leq C n (1 + t^2) \ex{t^2/2}, 
    \end{align}
    where $C > 0$ is a constant depending only on $\beta$. 
    \label{lem:5.2}
\end{lemma}

Recall that for each $1 \leq i \leq n$, given $\left\{ X_j, j \neq i \right\}$,  
$X_i'$ is conditionally independent of $X_i$  with the conditional distribution of $X_i$. 
Also, recall the normalizing constant $Z_n = \E \exp\bigl\{ \beta(\xi_1 + \dots + \xi_n)^2/(2n) \bigr\}$. 
\begin{lemma} 
    \label{lem:c5.3}
    For $0 < \beta < 1$, we have 
    \begin{align}
        \label{eq:c5.3-a}
        1 \leq Z_n \leq C, 
    \end{align}
	and for $n > 4 \{1,\beta /(1 - \beta)\}$ and $0 \leq t \leq \sqrt{n}$, 
    \begin{align}
        \label{eq:c5.3-c} 
        \E \{|X_i|^{6} \ex{t W}\} & \leq C \ex{t^2/2}, \\
        \label{eq:c5.3-d} 
        \E \{|X_i'|^6 \ex{t W}\} & \leq C \ex{t^2 / 2}. 
    \end{align}
\end{lemma}

The following lemma provides an upper bound of $\bigl\lvert \E \bigl\{  (X_i^k -\mu_k )(X_j^k - \mu_k) \ex{t W}  \bigr\} \bigr\rvert$, whose proof is similar to Lemma 5.7 of \cite{Sh19B}.  
\begin{lemma}
    \label{lem:5.7}
	For $i \neq j \in [n]$ and $k = 1 , 2, 3$, we have for $n > 16\max\{ 1, \beta/(1 - \beta)\}$ and $0 \leq t \leq \sqrt{n}$, 
    \begin{align*}
		\bigl\lvert \E \bigl\{  (X_i^k -\mu_k )(X_j^k - \mu_k) \ex{t W}  \bigr\} \bigr\rvert \leq C n^{-1} (1 + t^2) \ex{t^2/2}, 
    \end{align*}
	where $\mu_k = \E \xi^k$. 
\end{lemma}
The details of proofs of \cref{lem:c5.1,lem:5.2,lem:c5.3,lem:5.7} are put in the Supplementary Material \citep{Zh21S}.

Now we are ready to prove \cref{pro:cw1}. 
\begin{proof}
    [Proof of \cref{pro:cw1}]
    Let $\xi, \xi_1, \dots, \xi_n$ be i.i.d.\ random variables with probability measure $\rho$. Let $\E_{\xi}$ denote the expectation with respect to $\xi$ conditional on other random variables. Recall that $\bar{X} = (X_1 + \dots + X_n)/n$. For each $i \in
	[n]$, let $\bar{X}_i = \bar{X} - X_i/n$. In what follows, we fix $n > 16 \max\{1,\beta / (1 - \beta)\}$ and $0 \leq t \leq \sqrt{n}$. Again, let $\alpha_n(t) = n^{-1/2}(1  - \beta)^{1/2} t$. We now prove \crefrange{eq:c5.3-b}{eq:p1.2-b} one by one.

	{\medskip\noindent\it (i). Proof of \cref{eq:c5.3-b}.}
	By \cref{eq:cw-c2}, \eqref{eq:c5.2-a} and \eqref{eq:c5.3-a}, we have 
    \begin{align*}
		\E \ex{t W} & =  \frac{1}{Z_n} \E \exp \biggl( \frac{\beta}{2n} T_n^2 + \alpha_n(t) T_n \biggr) \leq C \ex{t^2/2}.
    \end{align*}
	{\medskip\noindent\it (ii). Proof of \cref{eq:p1.1-b}.}
	Let $T_n = \xi_1 + \dots \xi_n$. 
	By \cref{eq:cw-c2,eq:c5.2-b,eq:c5.3-a}, we have 
	\begin{equ}
		\E \bigl\{ \lvert \bar X \rvert^2 \ex{t W} \bigr\}
		& = \frac{1}{n^2 Z_n} \E \biggl\{ T_n^2 \exp \biggl( \frac{\beta}{2n} T_n^2 + \alpha_n(t) T_n \biggr) \biggr\} \leq C n^{-1} (1 + t^2) e^{t^2/2}.
        \label{eq-p1.1-b0}
	\end{equ}
	By H\"older's inequality, \cref{eq:c5.3-b,eq-p1.1-b0}, we have 
	\begin{align*}
		\E \bigl\{ \lvert \bar X \rvert \ex{t W} \bigr\}
		& \leq (\E \ex{tW})^{1/2} \bigl( \E \bigl\{ \lvert \bar X \rvert^2 \ex{t W} \bigr\} \bigr)^{1/2} \leq C n^{-1/2} (1 + t) \ex{t^2/2}.
	\end{align*}

	{\medskip\noindent\it (iii). Proof of \cref{eq:p1.1-a}.}
    By the definition of $(S_n, S_n')$, we have 
    \begin{equ}
        \label{eq:p1.1-01}
        \conep{S_n - S_n'}{\mathcal{X}} 
        & = \frac{1}{n} \sum_{i = 1}^{n} \conep{X_i - X_i'}{\mathcal{X}}  = \bar{X} - \frac{1}{n} \sum_{i = 1}^{n} \frac{ \E_{\xi} \{\xi \ex{ \frac{\beta \xi^2}{2n} + \beta \bar{X}_i \xi }\}  }{  \E_{\xi} \{\ex{ \frac{\beta \xi^2}{2n} + \beta \bar{X}_i \xi }\}  }.
    \end{equ}
    Observe that 
    \begin{equ}
        \label{eq:p1.1-02} 
        \frac{ \E_{\xi} \{\xi \ex{ \frac{\beta \xi^2}{2n} + \beta \bar{X}_i \xi }\}  }{  \E_{\xi} \{\ex{ \frac{\beta \xi^2}{2n} + \beta \bar{X}_i \xi }\}  } = h(\bar{X}_i) + r_{1i}, 
    \end{equ}
    where 
    \begin{align*}
		h(s) & =  \frac{\E \{\xi \ex{ \beta s \xi }\} }{\E \ex{\beta s \xi}} , & 
        r_{1i} & = \frac{ \E_{\xi} \{\xi \ex{ \frac{\beta \xi^2}{2n} + \beta \bar{X}_i \xi }\}  }{  \E_{\xi} \{\ex{ \frac{\beta \xi^2}{2n} + \beta \bar{X}_i \xi }\}  } - \frac{ \E_{\xi} \{\xi \ex{  \beta \bar{X}_i \xi }\}  }{  \E_{\xi} \{\ex{   \beta \bar{X}_i \xi }\}  } . 
    \end{align*}
	Note that $n > 8 \{1,\beta / (1 - \beta)\}$, and thus $\beta/(2n) \leq 1/16$. Moreover, it is easy to see that $x^3 \leq 10 e^{x^2/8}$ for all $x > 0$. Therefore, 
	\begin{align*}
		\bigl\lvert \E_{\xi} \{ \xi e^{\frac{\beta \xi^2}{2n} + \beta \bar X_i \xi} \} - \E_{\xi} \{ \xi e^{\beta \bar X_i \xi} \} \bigr\rvert
		& \leq n^{-1}\E \bigl\{ \lvert \xi \rvert^3 e^{ \frac{\beta \xi^2}{2n} + \beta \bar X_i \xi } \bigr\}\\
		& \leq n^{-1} e^{\beta \bar X_i^2} \E \bigl\{ \lvert \xi \rvert^3 e^{ \xi^2/16 +  \xi^2/4 } \bigr\}\\
		& \leq C n^{-1} e^{\beta\bar X_i^2} \E e^{ 3 \xi^2/8 }\\
		& \leq C n^{-1} e^{\beta \bar X_i^2},
	\end{align*}
	where we used H\"older's inequality in the third line and 
	where we used \cref{lem:c5.1} in the last line. Similarly, 
	\begin{align*}
		\bigl\lvert \E_{\xi} \{ \xi e^{\frac{\beta \xi^2}{2n} + \beta \bar X_i \xi} \} - \E_{\xi} \{ \xi e^{\beta \bar X_i \xi} \} \bigr\rvert
		& \leq C n^{-1} e^{\beta \bar X_i^2}.
	\end{align*}
	As $\E \xi = 0$, it follows from the Jensen inequality that $\E_{\xi}e^{\beta s \xi} \geq 1$ for all $s \in \R$. Hence, 
    \begin{equ}
        |r_{1i}| \leq C n^{-1} \ex{\beta \bar{X}_i^2}.  
        \label{eq-p5.5-r1i}
    \end{equ}
    By Taylor's expansion, 
    \begin{equ}
        \label{eq:p1.1-03} 
        h(\bar{X}_i) & = \beta \bar{X} - \frac{\beta}{n} X_i + \int_{0}^{\bar{X}_i} h'' (t) (\bar{X}_i - t) \dd t.
    \end{equ}
    By \cite[][Eq. (5.41)]{Sh19B}, 
    \begin{equ}
        \label{eq:p1.1-07} 
        \biggl\lvert \int_{0}^{\bar{X}_i} h'' (t) (\bar{X}_i - t) \dd t \biggr\rvert \leq C |\bar{X}_i|^2 \ex{ \beta \bar{X}_i^2 }.
    \end{equ}
    It follows from \cref{eq:p1.1-02,eq-p5.5-r1i,eq:p1.1-03,eq:p1.1-07} that 
	\begin{equ}
		r_{2i} \coloneqq \biggl\lvert \frac{ \E_{\xi} \{\xi \ex{ \frac{\beta \xi^2}{2n} + \beta \bar{X}_i \xi }\}  }{  \E_{\xi} \{\ex{ \frac{\beta \xi^2}{2n} + \beta \bar{X}_i \xi }\}  } - \beta \bar X \biggr\rvert \leq \frac{\beta}{n} \lvert X_i \rvert + C \lvert \bar X_i
		\rvert^2 e^{\beta \bar X_i^2}. 
        \label{eq-p1.1-075}
	\end{equ}
	By \cref{eq:p1.1-01,eq-p1.1-075}, 
    \begin{equ}
        \label{eq:p1.1-08} 
        |R_1 | \leq \frac{C }{n} \sum_{i = 1}^{n} \biggl\{ \beta n^{-1}|X_i| + |\bar{X}_i|^2 \ex{ \beta \bar{X}_i^2 }  \biggr\}.
    \end{equ}
	Next we prove the bound of $\E \lvert R_1 \rvert \ex{t W}$. 
	Note that $\alpha_n(t) \leq 1$ for $0 \leq t \leq \sqrt{n}$, and by H\"older's inequality and by \cref{lem:c5.1},
	\begin{equ}
		\E e^{\beta \xi_1^2/2 + \alpha_n(t) \xi_1}
		& \leq e^{\alpha_n(t)^2/(1 - \beta)}\E e^{ \beta \xi_1^2/2 + (1 - \beta) \xi^2/4 } \leq C. 
        \label{eq-p1.1-085}
	\end{equ}
	Let $M_1 = S_n - \xi_1$. 
By \cref{lem:5.2} with $m = n - 1$ and by \cref{eq-p1.1-085}, 
    \begin{equ}
        \label{eq:p1.1-05} 
        \E \{|\bar{X}_i|^2 \ex{ \beta \bar{X}_i^2 + t W}\}
		& \leq \frac{1}{n^2 Z_n} \E \biggl\{M_1^2 \exp \biggl( \Bigl( \frac{\beta}{ 2n } + \frac{\beta}{n^2}  \Bigr) M_1^2 + \frac{\beta}{2}\xi_1^2 + \alpha_n(t) (\xi_1 + M_1) \biggr)\biggr\} \\
		& =  \frac{1}{n^2 Z_n} \E e^{\beta \xi_1^2 / 2 + \alpha_n(t) \xi_1}\E \biggl\{M_1^2 \exp \biggl( \Bigl( \frac{\beta}{ 2n } + \frac{\beta}{n^2}  \Bigr) M_1^2  + \alpha_n(t) M_1 \biggr)\biggr\}\\
		& \leq C n^{-1}(1 + t^2) \ex{t^2/2}.
    \end{equ}
    By \eqref{eq:c5.3-c}, 
    \begin{equ}
        \label{eq:p1.1-06} 
        \E \{|X_i| \ex{ t W }\} \leq C \ex{t^2/2}. 
    \end{equ}
    Combining \cref{eq:p1.1-05,eq:p1.1-06}, 
    we complete the proof of \cref{eq:p1.1-a}.

	{\medskip\noindent\it (iv). Proof of \cref{eq:p1.2-a}.}
    Observe that 
    \begin{equ}
        \label{eq:p1.2-01}
        \conepbc{ (S_n - S_n')^2 }{\mathcal{X}} = 2 + R_2 + R_3 + R_4, 
    \end{equ}
    where 
    \begin{align*}
        R_2 = \frac{1}{n} \sum_{i = 1}^{n} (X_i^2 - 1), \quad 
        R_{3} = - \frac{1}{n} \sum_{i = 1}^{n} \frac{ 2 X_i \E_{\xi} \{\xi \ex{ \frac{\beta \xi^2}{2n} + \beta \bar{X}_i \xi }\}  }{\E_{\xi} \{ \ex{ \frac{\beta \xi^2}{2n} + \beta \bar{X}_i \xi }\} }, \quad 
        R_4 = \frac{1}{n} \sum_{i = 1}^{n} \frac{ \E_{\xi} \{\xi^2 \ex{ \frac{\beta \xi^2}{2n} + \beta \bar{X}_i \xi }\}  }{\E_{\xi} \{\ex{ \frac{\beta \xi^2}{2n} + \beta \bar{X}_i \xi }\} } - 1. 
    \end{align*}
For $R_2$, applying  \cref{lem:5.7} with $k = 2$, and by \cref{eq:c5.3-b} and the Cauchy inequality  that 
\begin{equ}
    \label{eq:p1.2-02} 
    \E \{|R_2| \ex{t W}\} 
    & \leq \bigl( \E \ex{tW} \bigr)^{1/2} \bigl( \E \{|R_2|^2 \ex{t W}\} \bigr)^{1/2}  \leq C n^{-1/2} (1 + t) \ex{t^2/2}.
\end{equ}
For $R_3$, note that by \eqref{eq-p1.1-075}, 
\begin{align*}
|R_3| & \leq  2\beta \bar X^2 +  \frac{2}{n}\sum_{i = 1}^{n} |X_i|  r_{2i} \leq 2 \beta \bar X^2 + \frac{2}{n^2} \sum_{i = 1}^n X_i^2 + C n^{-1} \sum_{i = 1}^n \lvert X_i \bar X_i^2 \rvert e^{\beta \bar X_i^2}.
\end{align*}
Similar to \cref{eq:p1.1-05}, we have 
\begin{align}
		\E \{|X_i \bar{X}_i^2| \ex{ \beta \bar{X}_i^2 + t W}\}
			& \leq  \frac{1}{n^2 Z_n} \E \{|\xi_1|e^{\beta \xi_1^2 / 2 + \alpha_n(t) \xi_1} \}\E \biggl\{M_1^2 \exp \biggl( \Bigl( \frac{\beta}{ 2n } + \frac{\beta}{n^2}  \Bigr) M_1^2  + \alpha_n(t) M_1 \biggr)\biggr\} \nonumber\\
			& \leq C n^{-1}(1 + t^2) \ex{t^2/2}.
        \label{eq-p1.1-21}
\end{align}
By \cref{eq-p1.1-b0,eq:c5.3-c,eq-p1.1-21}, we obtain 
\begin{equ}
	\E |R_3| & \leq C n^{-1} (1 + t^2)\ex{t^2/2}.
    \label{eq:p1.2-03}
\end{equ}
For $R_4$, note that 
\begin{align*}
    \frac{ \E_{\xi} \{\xi^2 \ex{ \frac{\beta \xi^2}{2n} + \beta \bar{X}_i \xi }\}  }{\E_{\xi} \{\ex{ \frac{\beta \xi^2}{2n} + \beta \bar{X}_i \xi }\} } - 1  = \frac{ \E_{\xi} \{(\xi^2 - 1)  \ex{ \frac{\beta \xi^2}{2n} + \beta \bar{X}_i \xi }\}  }{\E_{\xi} \{ \ex{ \frac{\beta
	\xi^2}{2n} + \beta \bar{X}_i \xi }\} } = \frac{ \E_{\xi} \{(\xi^2 - 1)  \ex{  \beta \bar{X}_i \xi }\}  }{\E_{\xi} \{ \ex{  \beta \bar{X}_i \xi }\} } + r_{3i}, 
\end{align*}
where 
\begin{align*}
    r_{3i} = \frac{ \E_{\xi} \{(\xi^2 - 1) \ex{ \frac{\beta \xi^2}{2n} + \beta \bar{X}_i \xi }\}  }{  \E_{\xi} \{\ex{ \frac{\beta \xi^2}{2n} + \beta \bar{X}_i \xi }\}  } - \frac{ \E_{\xi} \{ (\xi^2 - 1) \ex{  \beta \bar{X}_i \xi }\}  }{  \E_{\xi} \{\ex{   \beta \bar{X}_i \xi }\}  } .
\end{align*}
Similar to \eqref{eq-p5.5-r1i}, we have 
\begin{math}
    |r_{3i}| \leq C n^{-1} e^{\beta \bar X_i^2}.
\end{math}
Applying \cref{eq:cw-c4} with $t = \pm 1$ implies $\E e^{|\xi|} \leq \E e^{ \xi } + \E e^{- \xi} \leq 2 e^{1/2}$, then 
\begin{math}
	\E \lvert \xi \rvert^3 \leq 1.4 \E e^{|\xi|} \leq 15.
\end{math}
Since $\E \xi = 0$, it follows that $ \E_{\xi} \{ \ex{  \beta \bar{X}_i \xi }\} \geq 1$ and
\begin{align*}
    \bigl\lvert \E_{\xi} \{(\xi^2 - 1)  \ex{  \beta \bar{X}_i \xi }\}\bigr\rvert 
    & \leq \bigl\lvert \E \{\xi^2 - 1  \} \bigr\rvert + \bigl\lvert \beta \bar{X}_i \E \{\xi (\xi^2 - 1)  \}  \bigr\rvert  + C \bar{X}_i^2  \E_{\xi} \{\lvert (\xi^2 - 1)\xi_i^2 \rvert   \ex{  \beta |\bar{X}_i \xi| }\} \\
	& \leq C \lvert \bar{X}_i \rvert + C \bar{X}_i^2 \ex{ \beta \bar{X}_i^2}.
\end{align*}
Therefore, 
\begin{align*}
	|R_4| \leq C n^{-1} \sum_{i = 1}^n \Bigl\{ \lvert \bar X_i \rvert + (n^{-1} + \bar X_i^2) \ex{\beta \bar X_i^2} \Bigr\}.
\end{align*}
By \cref{eq:p1.1-b,eq:c5.3-c}, we have 
\begin{align*}
	\E \bigl\{ \lvert \bar X_i \rvert e^{tW} \bigr\}
	& \leq \E \bigl\{ \lvert \bar X \rvert e^{tW} \bigr\} + \frac{1}{n}\E \{ \lvert X_i \rvert e^{tW} \} \leq C n^{-1/2}( 1 + t ) e^{t^2/2}.
\end{align*}
Similar to \eqref{eq:p1.1-05}, 
\begin{align*}
	\E  \bigl\{(n^{-1} + \bar X_i^2) \ex{\beta \bar X_i^2 + t W}\bigr\} 
	& \leq C n^{-1}(1 + t^2) \ex{t^2/2}. 
\end{align*}
Therefore,
\begin{equ}
    \label{eq:p1.2-05}
	\E \{|R_4| \ex{tW}\} \leq C n^{-1/2} (1 + t) \ex{t^2/2}.
\end{equ}
This completes the proof of \cref{eq:p1.2-a} by combining \crefrange{eq:p1.2-01}{eq:p1.2-05}.

{\medskip\noindent\it (v). Proof of \cref{eq:p1.2-b}.}
Observe that 
\begin{align*}
	\conepb{ (S_n - S_n')^3 }{\mathcal{X}}
	= R_5 + R_6 + R_7 + R_8, 
\end{align*}
where 
\begin{align*}
	R_5 & = \frac{1}{n} \sum_{i = 1}^{n} (X_i^3 - \E \xi^3), & 
	R_6 & = - \frac{1}{n} \sum_{i = 1}^{n} \frac{3 X_i^2 \E_{\xi} \{ \xi e^{ \frac{\beta \xi^2}{2n} + \beta \bar X_i \xi } \}}{\E_{\xi} \{ e^{ \frac{\beta \xi^2}{2n} + \beta \bar X_i \xi } \}}, \\
	R_7 & = \frac{1}{n} \sum_{i = 1}^{n} \frac{3 X_i \E_{\xi} \{ \xi^2 e^{ \frac{\beta \xi^2}{2n} + \beta \bar X_i \xi } \}}{\E_{\xi} \{ e^{ \frac{\beta \xi^2}{2n} + \beta \bar X_i \xi } \}}, & 
	R_8 & = - \frac{1}{n} \sum_{i = 1}^{n} \biggl(\frac{\E_{\xi} \{ \xi^3 e^{ \frac{\beta \xi^2}{2n} + \beta \bar X_i \xi } \}}{\E_{\xi} \{ e^{ \frac{\beta \xi^2}{2n} + \beta \bar X_i \xi } \}} - \E \xi^3\biggr). 
\end{align*}
Similar to \emph{(iv)}, the inequality \cref{eq:p1.2-b} holds.
\end{proof}

\section*{Acknowledgement}
The author would like to thank the associate editor and two referees for their valuable comments which led to substantial improvement in the presentation of this paper. 
The author would also like to thank Qi-Man Shao for his comments. 
Part of the paper was completed during the period of my visit at the 
Chinese University of Hong Kong. It was partially supported by Hong Kong Research Grants Council GRF 14304917.
This research was also partially supported by the Australian Research Council Centre of Excellence for Mathematical and Statistical Frontiers  CE140100049 and by Singapore Ministry of Education Academic Research Fund MOE
2018-T2-076.


\end{document}